\documentclass[
11pt,                          
final,                         
english                        
]{article}

\usepackage[english]{babel}    
\usepackage{amsmath}           
\usepackage{amssymb}           
\usepackage[utf8]{inputenc}    
\usepackage[T1]{fontenc}       
\usepackage[a4paper]{geometry} 
\usepackage{setspace}          
\usepackage{bbm}               
\usepackage{mathrsfs}          
\usepackage{mathtools}         
\usepackage{stmaryrd}          
\usepackage[thmmarks,hyperref]{ntheorem} 
\usepackage{exscale}           
\usepackage[sort]{cite}        
\usepackage{xspace}            
\usepackage[normalem]{ulem}    
\usepackage[activate={true,nocompatibility},
            final,tracking=true,kerning=true,
            spacing=true,factor=1100,
            stretch=10,shrink=10,expansion=false
           ]{microtype}
\microtypecontext{spacing=nonfrench}  
\usepackage[hyperref]{xcolor}
\definecolor{mydarkblue}{RGB}{0,0,155}
\usepackage[backref=page,      
           final=true,         
           colorlinks=true,    
           allcolors=mydarkblue,  
           hypertexnames=false,
           plainpages=false,              
           pdfpagelabels=true,            
           pdfencoding=auto,              
           unicode=true                   
                      ]{hyperref}         
\usepackage[shortcuts]{extdash}  

\author{
  \textbf{Matthias Schötz}
  \thanks{Boursier de l'ULB, \href{mailto:Matthias.Schotz@ulb.ac.be}{\texttt{Matthias.Schotz@ulb.ac.be}}.
  This work was supported by the Fonds de la Recherche Scientifique (FNRS) and the Fonds Wetenschappelijk
  Onderzoek - Vlaaderen (FWO) under EOS Project n$^0$30950721.}\\
  Département de Mathématique\\
  Université libre de Bruxelles
}

\makeatletter

\newcommand{\refitem}[1] {\textit{\ref{#1}.)}}

\numberwithin{equation}{section}

\renewcommand{\arraystretch}{1.2}

\let\originalleft\left
\let\originalright\right
\renewcommand{\left}{\mathopen{}\mathclose\bgroup\originalleft}
\renewcommand{\right}{\aftergroup\egroup\originalright}

\newcommand{\lemmachairxname}{Lemma}
\newcommand{\propositionchairxname}{Proposition}
\newcommand{\theoremchairxname}{Theorem}
\newcommand{\corollarychairxname}{Corollary}
\newcommand{\definitionchairxname}{Definition}
\newcommand{\examplechairxname}{Example}
\newcommand{\proofchairxname}{Proof}

\StartBabelCommands{english}{extras}
\SetString{\lemmachairxname}{Lemma}
\SetString{\propositionchairxname}{Proposition}
\SetString{\theoremchairxname}{Theorem}
\SetString{\corollarychairxname}{Corollary}
\SetString{\definitionchairxname}{Definition}
\SetString{\examplechairxname}{Example}
\SetString{\proofchairxname}{Proof}
\EndBabelCommands

\theoremheaderfont{\normalfont\bfseries}
\theorembodyfont{\itshape}
\newtheorem{lemma}{\lemmachairxname}[section]
\newtheorem*{lemmaNN}{\lemmachairxname}
\newtheorem{proposition}[lemma]{\propositionchairxname}
\newtheorem{theorem}[lemma]{\theoremchairxname}
\newtheorem*{theoremNN}{\theoremchairxname}
\newtheorem{corollary}[lemma]{\corollarychairxname}
\newtheorem{definition}[lemma]{\definitionchairxname}

\theorembodyfont{\rmfamily}

\newtheorem{example}[lemma]{\examplechairxname}

\def\theorem@checkbold{}

\theoremheaderfont{\normalfont\bfseries}
\theorembodyfont{\normalfont}
\theoremstyle{nonumberplain}
\theoremseparator{:}
\theoremsymbol{\hbox{$\boxempty$}}
\newtheorem{proof}{\proofchairxname}
\theoremsymbol{\hbox{$\triangledown$}}

\newcommand{\Unit}           {\mathbbm{1}}
\newcommand{\id}              {\mathrm{id}}
\newcommand{\argument}       {\ignorespaces{\,\cdot\,}\ignorespaces}
\newcommand{\D}              {\mathop{}\!\mathrm{d}}
\DeclarePairedDelimiter{\abs}{\lvert}{\rvert}
\DeclarePairedDelimiter{\norm}{\lVert}{\rVert}
\newcommand{\Stetig}         {\mathscr{C}}
\newcommand{\cc}[1]          {\overline{{#1}}}
\newcommand{\I}              {\mathrm{i}}

\DeclarePairedDelimiter{\ordinaryIP}{\langle}{\rangle}
\DeclarePairedDelimiter{\ordinarySet}{\{}{\}}

\newcommand{\RR}{\mathbbm{R}}
\newcommand{\CC}{\mathbbm{C}}
\newcommand{\NN}{\mathbbm{N}}
\newcommand{\skal}[3][]{\ordinaryIP[#1]{\,#2 \,#1|\, #3\,}}
\newcommand{\set}[3][]{\ordinarySet[#1]{\,#2 \;#1|\; #3\,}}
\newcommand{\cl}{\mathrm{cl}}
\newcommand{\acl}{{a\textup{-}\mathrm{cl}}}
\newcommand{\seminorm}[3][]{\norm[#1]{#3}_{#2}}
\newcommand{\Hilb}{\mathfrak{H}}
\newcommand{\neu}[1]{\uline{#1}}
\newcommand{\mult}[1]{{{\mathrm M}_{#1}}}

\makeatother

\title{
Stone-Weierstraß Theorems\\
for Riesz Ideals of Continuous Functions%
}

\date{February 2019}

\begin{document}
\begin{onehalfspace}
\allowdisplaybreaks
\maketitle

\begin{abstract}
  Notions of convergence and continuity specifically adapted 
  to Riesz ideals $\mathcal{I}$ of the space of continuous real-valued functions
  on a Lindelöf locally compact Hausdorff space are given, and used to prove 
  Stone-Weierstraß-type theorems for $\mathcal{I}$.
  As applications, sufficient conditions are discussed that guarantee that various types of 
  positive linear maps on $\mathcal{I}$ are uniquely determined by their restriction 
  to various point-separating subsets of $\mathcal{I}$. A very special case of this
  is the characterization of the strong determinacy of moment problems,
  which is rederived here in a rather general setting and without making use of spectral theory.
\end{abstract}

\section{Introduction}
\label{sec:Introduction}
Riesz ideals $\mathcal{I}$ of the lattice-ordered algebra $\Stetig(X)$ of continuous real-valued functions 
on a locally compact Hausdorff space $X$ arise naturally as the continuous functions 
which are integrable with respect to a given positive Borel measure $\mu$ on $X$. The 
question of determinacy of a moment problem can be phrased as the question, whether for
two positive Borel measures $\mu,\mu'$ on $X$, whose integrals coincide 
on a given point-separating subset $S$ of $\mathcal{I}$, the integrals even coincide on whole $\mathcal{I}$
(which, under suitable regularity assumptions on $\mu$ and $\mu'$, implies that $\mu=\mu'$).

As an example, consider the determinacy of the Hamburger moment problem: Given two
positive Borel measures $\mu$ and $\mu'$ on $\RR$ such that 
$\int_\RR p\,\D \mu = \int_\RR p\,\D \mu' \in \RR$ holds for all polynomial
functions $p$ on $\RR$, do $\mu$ and $\mu'$ coincide? If $\mu$ has compact support,
then so does $\mu'$ (see e.g. \cite[Prop.~4.1]{schmuedgen:TheMomentProblem}), and 
in this case both integrals describe linear functions from
$\Stetig(\RR)$ to $\RR$ that are continuous with respect to the Fréchet topology
of $\Stetig(\RR)$, i.e. the locally convex topology of uniform convergence on all compact subsets of $\RR$.
As the polynomial functions are dense in $\Stetig(\RR)$ with respect to this topology
by the classical Stone-Weierstraß theorem, the integrals over $\mu$ and $\mu'$ coincide on
whole $\Stetig(\RR)$. However, if $\mu$ does not have compact support, then it is no longer
continuous in this topology (and only describes a linear functional on a true Riesz
ideal $\mathcal{I}$ of $\Stetig(\RR)$, as all positive linear functionals defined  on the 
whole Fréchet algebra $\Stetig(\RR)$ are necessarily continuous by \cite{ng.warner:ContinuityOfPositiveAndMultiplicativeFunctionals}). 
If one wants to keep the same straightforward line of reasoning as in the compactly
supported case, an adapted notion of convergence in $\mathcal{I}$ is needed.

This adapted notion of convergence of sequences, together with those of Cauchy sequences,
closedness, closure and continuity,
is given in Definition~\ref{definition:convergence} and is more order theoretic than
topological in nature. Proposition~\ref{proposition:Cauchy} and its Corollary~\ref{corollary:convergence}
then clarify the relation to the analogous notions coming from the locally convex topology of uniform convergence 
on all compact subsets, and essentially establish the completeness of $\mathcal{I}$.
The main results are Theorems~\ref{theorem:convergence}
and \ref{theorem:convergeAlg}, which are generalizations of the classical
Stone-Weierstraß theorem (an order theoretic and an algebraic one), and Theorem~\ref{theorem:autocont},
which complements these by a result about automatic continuity of certain positive linear maps.
Using the generalized Stone-Weierstraß theorems, a characterization of the (strong)
determinacy of moment problems is given in the final Theorem~\ref{theorem:determinacy}.
Here, a self-adjointness condition is needed to assure that Theorem~\ref{theorem:convergeAlg}
can be applied.

Acknowledgements:
I would like to thank Prof. K. Schmüdgen for some valuable hints and remarks.
\section{Preliminaries}
The natural numbers are $\NN = \{1,2,3,\dots\}$, $\NN_0 \coloneqq \NN\cup\{0\}$ and 
the sets of real and complex numbers are denoted by $\RR$ and $\CC$, respectively.

Throughout this note, $X$ is a non-empty locally compact Hausdorff space which
is additionally Lindelöf, i.e. every open cover of $X$ has a countable subcover.
By \cite[Chap.~5, Thms.~17--18]{kelley:GeneralTopology}, every locally compact
Hausdorff space is a Tychonoff space, so for every $x\in X$ and every open neighbourhood
$U$ of $x$ there exists a continuous function $f\colon X\to[0,1]$
such that $f(x) = 1$ and $f(y) = 0$ for all $y\in X\backslash U$. The assumption
that $X$ is Lindelöf has one important and well-known consequence:
\begin{lemmaNN} \label{lemma:exhaustion}
  The Lindelöf and locally compact Hausdorff space $X$ admits a countable exhaustion 
  by compact sets, i.e. there exists a sequence $(K_n)_{n\in \NN}$ of compact subsets
  of $X$ such that $X = \bigcup_{n\in \NN} K_n$ and which is strictly increasing in 
  the sense that $K_n$ is contained in the interior of $K_{n+1}$ for every $n\in \NN$.
\end{lemmaNN}
\begin{proof}
  For convenience of the reader, a proof is given:
  For every $x\in X$ there exists an open neighbourhood $U_x$ of $x$
  that has compact closure, and then also a continuous function $f_x \colon X\to [0,1]$
  fulfilling $f_x(x) = 1$ and $f_x(y) = 0$ for all $y\in X\backslash U_x$, so
  $x\in f_x^{-1}\big(]0,\infty[\big) \subseteq U_x$. As these open sets
  $f_x^{-1}\big(]0,\infty[\big)$ with $x$ ranging over $X$ cover whole $X$, there exists 
  a sequence $(x_m)_{m\in \NN}$ in $X$ such that $X=\bigcup_{m\in\NN} f_{x_m}^{-1}\big(]0,\infty[\big)$.
  For every $n\in \NN$, the subset $K_n \coloneqq \bigcup_{m=1}^n f_{x_m}^{-1}\big( [1/n, \infty[ \big)$
  of $X$ is a closed subsets of the compact closure of $\bigcup_{m=1}^n U_{x_m}$, hence compact.
  Moreover, it follows from $K_n \subseteq \bigcup_{m=1}^n f_{x_m}^{-1}\big( ]1/(n+1), \infty[ \big) \subseteq K_{n+1}$
  that the sequence $(K_n)_{n\in \NN}$ is strictly increasing. Finally,
  for every $\hat{x}\in X$ there exists an $m\in \NN$ such that $\hat{x}\in f_{x_m}^{-1}\big( ]0, \infty[ \big)$
  by construction of the sequence $(x_m)_{m\in \NN}$, and then also an $n\in \NN$, $n\ge m$,
  such that even $\hat{x}\in f_{x_m}^{-1}\big( [1/n, \infty[ \big)$, hence $\hat{x}\in K_n$.
\end{proof}
As a consequence, every compact subset $K'$ of $X$ is eventually contained in one
of the $K_n$ constructed above with sufficiently large $n\in \NN$, because their
interiors cover $K'$. Important examples of
Lindelöf spaces are second countable spaces, i.e. topological spaces whose
topology admits a countable base, see e.g. \cite[Chap.~1, Thm.~15]{kelley:GeneralTopology}.

If $P$ and $Q$ are partially ordered sets, then a map $\Phi \colon P\to Q$ is called 
\neu{increasing} (\neu{decreasing}) if $\Phi(p) \le \Phi(p')$ (or $\Phi(p) \ge \Phi(p')$, respectively) holds 
for all $p,p'\in P$ with $p\le p'$.

An \neu{ordered vector space} is a real vector space $V$ endowed with a partial order
$\le$ such that $\lambda u \le \lambda v$ and $u + w \le v + w$ hold for
all $\lambda \in {[0,\infty[}$ and all $u,v,w\in V$ with $u\le v$.
Moreover, $V$ is called \neu{Archimedean} if it has the following property: Whenever $v\le \epsilon w$ 
holds for two vectors $v\in V$, $w\in V^+$ and all $\epsilon \in {]0,\infty[}$, then $v \le 0$. 
For every ordered vector space $V$, the convex cone of \neu{positive
elements} is $V^+ \coloneqq \set{v\in V}{v\ge 0}$.
If $V$ and $W$ are two ordered vector spaces, then a linear map $\Phi\colon V\to W$
is increasing if and only if $\Phi(v) \in W^+$ for all $v\in V^+$. Such an increasing
linear map is called \neu{positive}. Every linear subspace $S$ of an ordered vector
space $V$ is again an ordered vector space with the order inherited from $V$.

A \neu{Riesz space} (or vector lattice) is an ordered vector space $\mathcal{R}$ in which 
$\abs{r} \coloneqq \sup \{r,-r\}$, the \neu{absolute value} of $r$,
exists for all $r\in \mathcal{R}$. In this case, supremum and infimum of all pairs of elements $r,s\in\mathcal{R}$
exist in $\mathcal{R}$, namely $r \vee s \coloneqq \sup \{r,s\} = (r+s)/2 + \abs{(r-s)/2}$
and $r \wedge s \coloneqq \inf \{r,s\} = (r+s)/2 - \abs{(r-s)/2}$, and $\vee$ as well as
$\wedge$ describe associative and commutative operations on $\mathcal{R}$. So it makes sense to write
$\bigvee_{n=1}^N r_n \coloneqq r_1 \vee \dots \vee r_N = \sup \{r_1,\dots,r_N\}$
for $N\in \NN$ and $r_1,\dots,r_N \in \mathcal{R}$, analogously for finite infima.
A \neu{Riesz subspace} of $\mathcal{R}$ is a linear
subspace $S\subseteq \mathcal{R}$ such that $\abs{s} \in S$ for all $s\in S$, where $\abs{s}$
denotes the absolute value of $s$ as an element of $\mathcal{R}$. Such a Riesz subspace $S$
is again a Riesz space whose finite suprema and infima coinicide with those in $\mathcal{R}$.
A \neu{Riesz ideal} of $\mathcal{R}$ is a linear subspace $\mathcal{I}$ of $\mathcal{R}$
with the property that, whenever $r \in \mathcal{R}$ and $s \in \mathcal{I}$ fulfil $\abs{r} \le \abs{s}$,
then $r\in \mathcal{I}$. A Riesz ideal is especially a Riesz subspace.
If $\mathcal{R}$ and $\mathcal{S}$ are two Riesz spaces, then a map $\Phi\colon \mathcal{R} \to \mathcal{S}$ is called
\neu{morphism of Riesz spaces} if $\Phi$ is linear and $\Phi(\abs{r}) = \abs[\big]{\Phi(r)}$
holds for all $r\in \mathcal{R}$. Such a morphism of Riesz spaces is especially positive. 
For details about Riesz spaces, see the standard textbooks on the subject, e.g.
\cite{luxemburg.zaanen:RieszSpacesI}.

The most important (Archimedean) Riesz space for the purpose of this note is $\Stetig(X)$, the set 
of all real-valued continuous functions on $X$, with pointwise operations and
pointwise order. Here $\big(f\vee g\big)(x) = \max\big\{f(x),g(x)\big\}$ and 
$\big(f\wedge g\big)(x) = \min\big\{f(x),g(x)\big\}$ for all $f,g\in \Stetig(X)$ and all $x\in X$.
The function on $X$ which is constant $1$ is denoted by $\Unit \in \Stetig(X)$.
A function $f\in \Stetig(X)$ is called \neu{uniformly bounded} if there exists a
$\lambda \in {[0,\infty[}$ such that $-\lambda \Unit \le f \le \lambda \Unit$.
The subset of all uniformly bounded functions in $\Stetig(X)$ is denoted by $\Stetig_b(X)$
and is a Riesz ideal and contained in every other Riesz ideal $\mathcal{I}$ of $\Stetig(X)$ with $\Unit \in \mathcal{I}$. 
Similarly, $\Stetig_c(X) \coloneqq \set[\big]{f\in \Stetig(X)}{ f^{-1}\big(\RR\backslash\{0\}\big)\text{ has compact closure} }$
is the Riesz ideal of $\Stetig(X)$ of \neu{compactly supported} functions.
Given $Y\subseteq X$, a sequence $(g_n)_{n\in \NN}$
in $\Stetig(X)$ and $\hat{g} \in \Stetig(X)$, then $(g_n)_{n\in \NN}$ is said to
\neu{converge uniformly on $Y$} against $\hat{g}$ if for all $\epsilon\in{]0,\infty[}$
there exists an $N \in \NN$ such that $\abs{\hat{g}(y) - g_n(y)} \le \epsilon$
holds for all $y\in Y$ and all $n\in \NN$ with $n\ge N$. Note that $\Stetig(X)$
is complete with respect to the locally convex topology of uniform convergence
on all compact subsets of $X$, and that it admits a countable fundamental system 
of continuous seminorms because of the existence of a countable exhaustion of $X$ 
by compact subsets. So $\Stetig(X)$ is a Fréchet space. However, we will be more
interested in Riesz ideals of $\Stetig(X)$ that are not closed in this topology.
Like for every other ordered vector space, one can discuss suprema and infima
of infinite sets in the Riesz space $\Stetig(X)$. But even if they exist,
they are not necessarily the pointwise ones.
If $(g_k)_{k\in \NN}$ is an increasing (or decreasing) sequence in $\Stetig(X)$ which has supremum (or infimum)
$\hat{g} \in \Stetig(X)$, then $\hat{g}$ is called the \neu{pointwise supremum} (or infimum)
if $\hat{g}(x) = \sup_{k\in \NN} g_k(x)$ (or $\hat{g}(x) = \inf_{k\in \NN} g_k(x)$, respectively)
for all $x\in X$. 

A subset $S$ of $\Stetig(X)$ is called \neu{point-separating}
if for all $x,y\in X$ with $x\neq y$ there exists an $s\in S$ fulfilling $s(x) \neq s(y)$. It is called
\neu{strictly point-separating} if for all $x,y\in X$ with $x\neq y$ there exists an $s\in S$ fulfilling 
$s(x) \neq 0$ and $s(y) = 0$. Moreover, $S$ is said to be \neu{nowhere-vanishing} if
for every $x\in X$ there exists an $s\in S$ with $s(x) \neq 0$. 
Note that a strictly point-separating subset $S$ of $\Stetig(X)$ is not necessarily
nowhere-vanishing if $X$ consists of only one point (for larger $X$, this is indeed true).
Moreover, a point-separating 
and nowhere-vanishing Riesz subspace $S$ of $\Stetig(X)$ is not strictly point-separating in general.
For example, consider the Riesz subspace 
$\set[\big]{f \in \Stetig\big([0,1]\big)}{f(0) = 2 f(1)}$ of $\Stetig\big([0,1]\big)$.
However, one easily sees that a point-separating linear subspace $S$ of $\Stetig(X)$ 
with $\Unit \in S$ is strictly point-separating and nowhere-vanishing.
Note also that whole $\Stetig(X)$ is strictly point-separating and nowhere-vanishing
because $X$ is especially a Tychonoff space, and that $\Stetig_c(X)$ is the smallest
Riesz ideal of $\Stetig(X)$ which is nowhere-vanishing: Indeed, if $\mathcal{I}$ is
a nowhere-vanishing Riesz ideal of $\Stetig(X)$ and $f\in\Stetig_c(X)$, then for every
$x \in X$ there exists a $g_x \in \mathcal{I}$ with $g_x(x) \neq 0$. One can even 
arrange $g_x \ge 0$ and $g_x(x) = 2$. Then the open subsets $g_x^{-1}\big( ]1,\infty[\big)$
of $X$ with $x$ ranging over $X$ cover the compact closure of $f^{-1}\big(\RR\backslash\{0\}\big)$,
hence there exist $x_1,\dots,x_N \in X$ with $N\in \NN$ such that $\hat{g} \coloneqq \sum_{n=1}^N g_{x_n} \in \mathcal{I}^+$
fulfils $\hat{g}(y) > 1$ for all $y\in f^{-1}\big(\RR\backslash\{0\}\big)$. As $f$ 
has compact support, it is uniformly bounded, so $\abs{f} \le \lambda \hat{g}$ holds 
for sufficiently large $\lambda \in {[0,\infty[}$ and thus $f\in \mathcal{I}$.

Finally, a \neu{ordered unital algebra} is a unital associative algebra $\mathcal{A}$ over the field $\RR$ which is
endowed with a partial order that turns $\mathcal{A}$ into an ordered vector space and in which
all squares are positive, i.e. $a^2 \in \mathcal{A}^+$ for all $a\in \mathcal{A}$. For example, 
$\Stetig(X)$, $\Stetig_b(X)$ and $\Stetig_c(X)$ are ordered algebras.
\section{Riesz Ideals of Continuous Functions} \label{sec:Main}
\begin{definition} \label{definition:convergence}
  Let $\mathcal{I}$ be a Riesz ideal of $\Stetig(X)$. Then a sequence $(g_n)_{n\in \NN}$
  in $\mathcal{I}$ is said to be \neu{strictly $\mathcal{I}$-convergent} against 
  an element $\hat{g} \in \mathcal{I}$, called the \neu{strict $\mathcal{I}$-limit}
  of $(g_n)_{n\in \NN}$, if there exists a decreasing sequence $(f_k)_{k\in \NN}$
  in $\mathcal{I}$ with pointwise infimum $0$ and such that for all $k\in \NN$ there 
  is an $N \in \NN$ with the property that $\abs{\hat{g}-g_n} \le f_k$ for all 
  $n\in \NN$ with $n\ge N$.

  Similarly, a sequence $(g_n)_{n\in \NN}$ in $\mathcal{I}$ is called a \neu{strict 
  $\mathcal{I}$-Cauchy sequence} if there exists a decreasing sequence $(f_k)_{k\in \NN}$
  in $\mathcal{I}$ with pointwise infimum $0$ and such that for all $k\in \NN$ 
  there is an $N \in \NN$ with the property that $\abs{g_n-g_{N}} \le f_k$ for all
  $n\in \NN$ with $n\ge N$.

  Moreover, a subset $S$ of $\mathcal{I}$ is called \neu{$\mathcal{I}$-closed} if
  it has the following property: If a sequence $(s_n)_{n\in \NN}$ in $S$ is 
  strictly $\mathcal{I}$\=/convergent against an element $\hat{s} \in \mathcal{I}$,
  then $\hat{s} \in S$.
  The \neu{$\mathcal{I}$-closure} of a subset $S$ of $\mathcal{I}$ is defined
  as the intersection of all $\mathcal{I}$\=/closed subsets of $\mathcal{I}$
  that contain $S$.

  Finally, if $V$ is an ordered vector space, then a positive linear map
  $\Phi\colon \mathcal{I} \to V$ is called \neu{strictly $\mathcal{I}$-continuous} 
  if for every decreasing sequence $(f_k)_{k\in \NN}$ in $\mathcal{I}$ that has
  pointwise infimum $0$, the decreasing sequence $\big( \Phi(f_k)\big)_{k\in \NN}$
  in $V$ has infimum $0$.
\end{definition}
The above notions of convergent and Cauchy sequences, closedness, closure and continuity 
are obviously adapted to the special case of a Riesz space of functions and are 
similar, but not identical, to the well-known notions related to $\sigma$\=/order
convergence (which uses all decreasing sequences with infimum $0$
instead of the ``special'' ones where the infimum is the pointwise one). 
For example, an increasing sequence in a Riesz ideal $\mathcal{I}$ of $\Stetig(X)$
with pointwise supremum in $\mathcal{I}$ is strictly $\mathcal{I}$\=/convergent against
this supremum. Note that
Definition~\ref{definition:convergence} also specifically refers to, and depends on, 
a chosen Riesz ideal:
\begin{example} \label{example:depends}
  Consider $X\coloneqq\NN$, $\mathcal{I} \coloneqq \Stetig(\NN)$, 
  $\mathcal{I}' \coloneqq \Stetig_b(\NN) \subseteq \mathcal{I}$ and the sequence 
  $(g_n)_{n\in \NN}$ in $\mathcal{I}'$ that is defined for $m,n\in \NN$ as $g_n(m) = 0$
  if $m\neq n$ and $g_n(n) = n$. This sequence converges against $0$ with respect
  to the locally convex topology of uniform convergence on all compact (i.e. finite) subsets of 
  $\NN$ and is also strictly $\mathcal{I}$\=/convergent against $0$:
  Choose e.g. the decreasing sequence $(f_k)_{k\in \NN}$ in $\mathcal{I}$ as 
  $f_k(m) \coloneqq 0$ for all $m\in \{1,\dots,k\}$ 
  and $f_k(m) \coloneqq m$ for all $m \in \NN \backslash \{1,\dots,k\}$, then 
  $\abs{0-g_n} = g_n \le f_k$ holds for all $n \in \NN$ with $n\ge k+1$.
  However $(g_n)_{n\in \NN}$ is not strictly $\mathcal{I}'$\=/convergent because 
  it is not even bounded from above by a function in $\mathcal{I}'$.
\end{example}
One very important class of examples of strictly $\mathcal{I}$\=/continuous 
positive linear maps are integrals:
\begin{example} \label{example:integrals}
  Let $\mathcal{I}$ be a Riesz ideal of $\Stetig(X)$. Then every positive Borel measure $\mu$
  on $X$ for which all functions in $\mathcal{I}$ are integrable yields a positive 
  linear map $\Phi_\mu \colon \mathcal{I} \to \RR$,
  \begin{equation*}
    f \mapsto \Phi_\mu (f) \coloneqq \int_X f \,\D \mu
  \end{equation*}
  which is strictly $\mathcal{I}$\=/continuous by the monotone convergence theorem for the 
  Lebesgue integral. Conversely, a positive linear map $\Phi\colon \mathcal{I} \to \RR$
  that is strictly $\mathcal{I}$\=/continuous fulfils the axioms of the Daniell integral.
\end{example}
Even though the notions of Definition~\ref{definition:convergence} were not introduced by 
means of a topology or uniform structure, it is not hard to see that they behave
-- at least to some extend -- as one would expect:

If $\mathcal{I}$ is a Riesz ideal of $\Stetig(X)$ and $(g_n)_{n\in \NN}$ a strictly
$\mathcal{I}$\=/convergent sequence in $\mathcal{I}$, then it is also a strict $\mathcal{I}$\=/Cauchy
sequence and its strict $\mathcal{I}$\=/limit $\hat{g}\in\mathcal{I}$ is given by
$\hat{g}(x) = \lim_{n\to \infty} g_n(x)$ for all $x\in X$, so the 
strict $\mathcal{I}$\=/limit is uniquely determined as the pointwise limit.
While it is easier to deal directly with the notions introduced in Definition~\ref{definition:convergence},
it is good to know that they are related to a topology on $\mathcal{I}$:
\begin{proposition} \label{proposition:topology}
  Let $\mathcal{I}$ be a Riesz ideal of $\Stetig(X)$, then the $\mathcal{I}$\=/closed subsets
  of $\mathcal{I}$ fulfil the axioms of the closed sets of a topological space, i.e.
  the empty set and whole $\mathcal{I}$, as well as the intersections of arbitrarily 
  many $\mathcal{I}$\=/closed subsets of $\mathcal{I}$ and the unions of finitely 
  many $\mathcal{I}$\=/closed subsets of $\mathcal{I}$ are $\mathcal{I}$\=/closed.
  
  Moreover, if $\tau \coloneqq \set[\big]{\mathcal{I}\backslash C}{C\subseteq \mathcal{I}\text{ is $\mathcal{I}$\=/closed}}$
  is the corresponding topology on $\mathcal{I}$, then the $\mathcal{I}$\=/closure of
  a subset $S$ of $\mathcal{I}$ is the closure with respect to $\tau$ and is
  especially $\mathcal{I}$\=/closed itself. Furthermore, every
  sequence $(g_n)_{n\in \NN}$ in $\mathcal{I}$ that is strictly $\mathcal{I}$\=/convergent
  against some $\hat{g}\in\mathcal{I}$ also converges against $\hat{g}$ with respect
  to the topology $\tau$.
\end{proposition}
\begin{proof}
  It is clear that the empty set and whole $\mathcal{I}$ as well as the intersections
  of arbitrarily many $\mathcal{I}$\=/closed sets are again $\mathcal{I}$\=/closed.
  If $C_1,\dots,C_M$ with $M\in \NN$ are $\mathcal{I}$\=/closed subsets of $\mathcal{I}$,
  then $C_1\cup \dots \cup C_M$ is again $\mathcal{I}$\=/closed because every sequence
  $(g_n)_{n\in \NN}$ in $C_1\cup \dots \cup C_M$ that is strictly $\mathcal{I}$\=/convergent
  against some $\hat{g}\in\mathcal{I}$ has a subsequence in at least one of the $C_m$
  with $m\in \{1,\dots,M\}$, and this subsequence is still strictly $\mathcal{I}$\=/convergent 
  against $\hat{g}$ so that $\hat{g} \in C_m \subseteq C_1\cup \dots \cup C_M$.

  The $\mathcal{I}$\=/closure of a subset $S$ of $\mathcal{I}$ coincides with the 
  closure with respect to $\tau$ by definition and is $\mathcal{I}$\=/closed itself
  because it is the intersection of $\mathcal{I}$\=/closed sets.
  If $(g_n)_{n\in \NN}$ is a
  sequence in $\mathcal{I}$ that is strictly $\mathcal{I}$\=/convergent against
  some $\hat{g}\in \mathcal{I}$, then it also converges against $\hat{g}$ with 
  respect to the topology $\tau$: Indeed, if this would not be true, then there
  would be an open neighbourhood $U\in \tau$ of $\hat{g}$ such that a subsequence
  of $(g_n)_{n\in \NN}$ would remain outside of $U$. But this subsequence in 
  $\mathcal{I}\backslash U$ would still be strictly $\mathcal{I}$\=/convergent against 
  $\hat{g} \in U$, which contradicts the $\mathcal{I}$\=/closedness 
  of $\mathcal{I}\backslash U$.
\end{proof}
Note that one has to be careful here: It is even unclear yet whether the topology
corresponding to the $\mathcal{I}$\=/closed sets is e.g. Hausdorff, or whether 
strict $\mathcal{I}$\=/convergence is also necessary for convergence of a sequence
in $\mathcal{I}$ with respect to this topology. This is the reason for using the
adjective ``strict'' in most of the notions introduced in Definition~\ref{definition:convergence}.
However, even without understanding the topology behind the $\mathcal{I}$\=/closure in detail,
one can derive some basic results using only the following properties of strictly 
$\mathcal{I}$\=/convergent sequences:

Assume that $\mathcal{I}$ is a Riesz ideal of $\Stetig(X)$ and $(g_n)_{n\in \NN}$
a sequence that is strictly $\mathcal{I}$\=/convergent against some $\hat{g}\in\mathcal{I}$, then:
\begin{itemize}
  \item For all $\lambda \in \RR$, the sequence $(\lambda g_n)_{n\in \NN}$ is 
  strictly $\mathcal{I}$\=/convergent against $\lambda \hat{g}$ because 
  $\abs{\lambda \hat{g} - \lambda g_n} \le \abs{\lambda} f_k$
  if $\abs{\hat{g} - g_n} \le f_k$ for some $f_k\in\mathcal{I}^+$.
  \item For all $h \in \mathcal{I}$, the sequence $(g_n+h)_{n\in \NN}$ clearly is 
  strictly $\mathcal{I}$\=/convergent against $\hat{g}+h$.
  \item The sequence $\big(\abs{g_n}\big)_{n\in \NN}$ is strictly
  $\mathcal{I}$\=/convergent against $\abs{\hat{g}}$ because $\abs[\big]{\abs{\hat{g}} - \abs{g_n}} \le f_k$
  if $\abs{\hat{g} - g_n} \le f_k$ for some $f_k\in\mathcal{I}^+$.
  \item If $\mathcal{I}$ is additionally a subalgebra of $\Stetig(X)$, then the 
  sequence $\big(g_n^2\big)_{n\in \NN}$ is strictly 
  $\mathcal{I}$\=/convergent against $\hat{g}^2$ because 
  $\abs[\big]{\hat{g}^2 - g_n^2} = \abs{\hat{g} - g_n} \abs[\big]{2\hat{g} + (g_n-\hat{g})} \le f_k(2\abs{\hat{g}} + f_k)$
  if $\abs{\hat{g} - g_n} \le f_k$ for some $f_k\in\mathcal{I}^+$.
\end{itemize}
\begin{proposition} \label{proposition:algebraPermanence}
  Let $\mathcal{I}$ be a Riesz ideal of $\Stetig(X)$ and $S$ a linear subspace of
  $\mathcal{I}$. Then the $\mathcal{I}$\=/closure of $S$ is again a linear subspace
  of $\mathcal{I}$. If $S$ is even a Riesz subspace of $\mathcal{I}$, then the 
  $\mathcal{I}$\=/closure of $S$ is also a Riesz subspace of $\mathcal{I}$. Finally,
  if $\mathcal{I}$ is additionally a subalgebra of $\Stetig(X)$ and if $S$ is
  a subalgebra of $\mathcal{I}$, then also the $\mathcal{I}$\=/closure of $S$ is
  a subalgebra of $\mathcal{I}$.
\end{proposition}
\begin{proof}
  Let $S^\cl$ denote the $\mathcal{I}$\=/closure of $S$. It is clear that $0\in S \subseteq S^\cl \subseteq \mathcal{I}$.
  
  Given $\lambda \in \RR$, then let $T_\lambda \coloneqq \set[\big]{t\in S^\cl}{\lambda t \in S^\cl}$.
  One can easily check that $T_\lambda$ is $\mathcal{I}$\=/closed, and due to $S\subseteq T_\lambda \subseteq S^\cl$
  it then follows immediately from the definition of the $\mathcal{I}$\=/closure $S^\cl$ of $S$
  that $T_\lambda=S^\cl$. This shows $\lambda \hat{s} \in S^\cl$ for all $\hat{s} \in S^\cl$ and all $\lambda \in \RR$.
  
  Next let $s\in S$ be given and define $T_s' \coloneqq \set[\big]{t\in S^\cl}{t + s\in S^\cl}$.
  Again one easily checks that $T_s'$ is $\mathcal{I}$\=/closed, and $S\subseteq T_s' \subseteq S^\cl$
  thus implies $T_s' = S^\cl$. So $\hat{s} + s \in S^\cl$ for all $\hat{s} \in S^\cl$ and 
  all $s\in S$. Furthermore, given $\hat{s} \in S^\cl$, then define $T''_{\hat{s}} \coloneqq 
  \set[\big]{t\in S^\cl}{\hat{s} +t \in S^\cl}$. Like before, $T''_{\hat{s}}$ is $\mathcal{I}$\=/closed
  and the previous considerations show that $S \subseteq T''_{\hat{s}} \subseteq S^\cl$,
  hence $T''_{\hat{s}} = S^\cl$. We conclude that $S^\cl$ is a linear subspace of $\mathcal{I}$.
  
  If $S$ is even a Riesz subspace of $\mathcal{I}$, let $T_{\abs{\argument}} \coloneqq \set[\big]{t\in S^\cl}{\abs{t}\in S^\cl}$.
  As $T_{\abs{\argument}}$ is $\mathcal{I}$\=/closed and as $S\subseteq T_{\abs{\argument}} \subseteq S^\cl$,
  we have $T_{\abs{\argument}} = S^\cl$ and thus $S^\cl$ is also Riesz subspace of $\mathcal{I}$.
  
  Similarly, if $\mathcal{I}$ is additionally a subalgebra of $\Stetig(X)$ and if $S$
  is a subalgebra of $\mathcal{I}$, then define $T_{\mathrm{sq}} \coloneqq \set[\big]{t\in S^\cl}{t^2 \in S^\cl}$.
  Again, $T_{\mathrm{sq}}$ is $\mathcal{I}$\=/closed and fulfils $S\subseteq T_{\mathrm{sq}} \subseteq S^\cl$,
  so $T_{\mathrm{sq}} = S^\cl$. It follows that $S^\cl$ is again a subalgebra of $\mathcal{I}$
  because $\hat{s}\hat{s}' = \frac{1}{4}\big( (\hat{s} + \hat{s}')^2 - (\hat{s} - \hat{s}')^2\big) \in S^\cl$
  for all $\hat{s},\hat{s}' \in S^\cl$.
\end{proof}
On first sight, strict $\mathcal{I}$\=/convergence might look like a rather weak notion
of convergence. However, it is even stronger than uniform convergence on all compact
subsets of $X$ due to Dini's theorem:
\begin{theoremNN}
  Let $K$ be a compact topological space and $(f_k)_{k\in \NN}$ a decreasing sequence 
  in $\Stetig(K)$ with pointwise infimum $0$, then the following holds:
  For all $\epsilon \in {]0,\infty[}$ there exists a $k\in \NN$ such that
  $f_k(x) \le \epsilon$ for all $x\in K$.
\end{theoremNN}
\begin{proof}
  For convenience of the reader, here is a proof of this classic result:
  Given $\epsilon \in {]0,\infty[}$, then for all $x\in K$ there exists a 
  $k_x\in \NN$ such that $f_{k_x}(x) \le \epsilon/2$ and a corresponding open 
  neighbourhood $U_x \coloneqq f_{k_x}^{-1}\big( ]-\infty,\epsilon[\big)$ of $x$.
  These open neighbourhoods $U_x$ for all $x\in K$ cover $K$, so there exists a finite set
  $x_1,\dots,x_M \in K$ with $M\in \NN_0$ such that $K \subseteq \bigcup_{m=1}^M U_{x_m}$.
  Define $k \coloneqq \max \{k_{x_1}, \dots, k_{x_M}\}$, then $f_{k}(x) < \epsilon$ 
  for all $x\in K$ because the sequence $(f_k)_{k\in \NN}$ is decreasing.
\end{proof}
Dini's theorem thus shows that directed pointwise convergence on compact topological
spaces implies uniform convergence and will also lead to an alternative description
of strict $\mathcal{I}$\=/Cauchy and strictly $\mathcal{I}$\=/convergent sequences:
\begin{lemma}   \label{lemma:convergence}
  Let $(e_n)_{n\in \NN}$ be a sequence in $\Stetig(X)^+$,
  which on all compact subsets of $X$ converges uniformly against $0$. Then the function
  $X \ni x\mapsto \hat{e}(x) \coloneqq \sup_{n\in \NN} e_n(x) \in {[0,\infty[}$
  is well-defined and continuous.
\end{lemma}
\begin{proof}
  Let $x\in X$ be given and let $K \subseteq X$ be a compact neighbourhood of $x$.
  The supremum $\hat{e}(x) = \sup_{n\in \NN} e_n(x) < \infty$ exists because there are only
  finitely many $n\in \NN$ with $e_n(x)\ge 1$ due to the convergence of the sequence 
  $(e_n)_{n\in \NN}$ against $0$.
  
  If $\hat{e}(x) > 0$, then there is an $N \in \NN$ such that $e_n(y) \le \hat{e}(x)/2$ holds
  for all $y\in K$ and all $n\in \NN$ with $n> N$. In this case, 
  $h \coloneqq \bigvee_{n=1}^N e_n \in \Stetig(X)$
  fulfils $h(y) = \hat{e}(y)$ for all $y\in K \cap h^{-1}\big(]\hat{e}(x)/2,\infty[\big)$, which
  is a neighbourhood of $x$, so $\hat{e}$ is continuous in $x$.
  
  If $\hat{e}(x) = 0$, let $\epsilon\in{]0,\infty[}$ be given. Then there exists an $N\in \NN$
  such that $e_n(y) \le \epsilon$ holds for all $y\in K$ and all 
  $n\in \NN$ with $n> N$. Define again
  $h \coloneqq \bigvee_{n=1}^N e_n \in \Stetig(X)$. Then
  $0 \le \hat{e}(y) \le \max\big\{h(y),\epsilon\big\}$ holds for all $y\in K$
  and thus $0 \le \hat{e}(y) \le \epsilon$ for all $y\in K \cap h^{-1}\big( ]-\infty,\epsilon[\big)$,
  which again is a neighbourhood of $x$. As a consequence, $\hat{e}$ is continuous in $x$ in this case as well.
\end{proof}
\begin{proposition} \label{proposition:Cauchy}
  Let $\mathcal{I}$ be a Riesz ideal of $\Stetig(X)$ and $(g_n)_{n\in \NN}$ a 
  sequence in $\mathcal{I}$, then the following is equivalent:
  \begin{enumerate}
    \item The sequence $(g_n)_{n\in \NN}$ is a strict $\mathcal{I}$\=/Cauchy sequence.
    \item There exists an element $b\in \mathcal{I}^+$ such that $\abs{g_n} \le b$ holds for all
    $n\in \NN$ and the sequence $(g_n)_{n\in \NN}$ is a Cauchy sequence with respect
    to the locally convex topology of uniform convergence on all compact subsets of $X$.
  \end{enumerate}
  If one, hence both of these statements are true, then the sequence $(g_n)_{n\in \NN}$
  converges against a limit $\hat{g}\in\mathcal{I}$ in the locally convex topology 
  of uniform convergence on all compact subsets of $X$, and is also strictly 
  $\mathcal{I}$\=/convergent with strict $\mathcal{I}$\=/limit $\hat{g}$.
\end{proposition}
\begin{proof}
  First assume that $(g_n)_{n\in \NN}$ is a strict $\mathcal{I}$\=/Cauchy sequence. Then
  there exists a decreasing
  sequence $(f_k)_{k\in \NN}$ in $\mathcal{I}$ with pointwise infimum $0$ and such that for 
  all $k\in \NN$ there is an $N \in \NN$ with the property that $\abs{g_n-g_N} \le f_k$ 
  for all $n\in \NN$ with $n\ge N$. Especially for $k=1$ and corresponding $N$ this implies
  that $b' \coloneqq f_1 \vee \bigvee_{n=1}^{N} \abs{g_n-g_N} \in \mathcal{I}^+$ fulfils
  $\abs{g_n-g_N} \le b'$ for all $n\in \NN$, so $b \coloneqq b' + \abs{g_N} \in \mathcal{I}^+$
  fulfils $\abs{g_n} \le \abs{g_n-g_N} + \abs{g_N} \le b$ for all $n\in \NN$.
  Moreover, Dini's theorem immediately shows that $(g_n)_{n\in \NN}$ is a Cauchy sequence with respect to the
  locally convex topology of uniform convergence on all compact subsets of $X$.
  
  Conversely, assume that there exists a $b\in \mathcal{I}^+$ such that $\abs{g_n} \le b$ for all
  $n\in \NN$ and that the sequence $(g_n)_{n\in \NN}$ is a Cauchy sequence with respect
  to the locally convex topology of uniform convergence on all compact subsets of $X$.
  Due to the completeness of $\Stetig(X)$, the sequence $(g_n)_{n\in \NN}$ 
  converges against some $\hat{g} \in \Stetig(X)$ in this topology,
  and this limit $\hat{g}$ necessarily fulfils $-b \le \hat{g} \le b$ because $-b \le g_n \le b$
  holds for all $n\in \NN$, so $\abs{\hat{g}} \le b$ and thus $\hat{g} \in \mathcal{I}$
  because $\mathcal{I}$ is a Riesz ideal of $\Stetig(X)$. 
  It remains to show that $(g_n)_{n\in \NN}$ is also strictly $\mathcal{I}$\=/convergent
  with strict $\mathcal{I}$\=/limit $\hat{g}$ (which especially implies that it
  is a strict $\mathcal{I}$\=/Cauchy sequence):
  
  For every $k\in \NN$, define the function $f_k \in \Stetig(X)^+$ as
  $f_k(x) \coloneqq \sup_{n\in \NN;\,n\ge k} \abs{\hat{g}(x)-g_n(x)}$
  for all $x\in X$, which is well-defined due to the previous Lemma~\ref{lemma:convergence}.
  As $\abs{\hat{g} - g_n} \le \abs{\hat{g}} + \abs{g_n} \le 2b$ holds for all $n\in \NN$,
  it follows that $f_k \le 2b$, so $f_k \in \mathcal{I}^+$. By construction, the resulting sequence 
  $(f_k)_{k\in \NN}$ in $\mathcal{I}$ is decreasing and fulfils $\abs{\hat{g}-g_n} \le f_k$ 
  for all $n\in\NN$ with $n \ge k$.
  It is also easy to see that $\inf_{k\in \NN} f_k(x) = 0$ for all $x\in X$, because
  for every $x\in X$ and every $\epsilon\in{]0,\infty[}$ there exists a $k\in \NN$
  such that $\abs{\hat{g}(x)-g_n(x)} \le \epsilon$ for all $n\in \NN$ with $n\ge k$,
  hence $f_k(x)\le \epsilon$.
\end{proof}
As strictly $\mathcal{I}$\=/convergent sequences are strict $\mathcal{I}$\=/Cauchy sequences 
and as strict $\mathcal{I}$\=/limits are uniquely determined as the pointwise ones,
this immediately yields:
\begin{corollary} \label{corollary:convergence}
  Let $\mathcal{I}$ be a Riesz ideal of $\Stetig(X)$ and $(g_n)_{n\in \NN}$ a 
  sequence in $\mathcal{I}$ as well as $\hat{g} \in \mathcal{I}$, then the following
  two statements are equivalent:
  \begin{enumerate}
    \item The sequence $(g_n)_{n\in \NN}$ is strictly $\mathcal{I}$\=/convergent
    against $\hat{g}\in\mathcal{I}$.
    \item There exists an element $b\in \mathcal{I}^+$ such that $\abs{g_n} \le b$ holds for all
    $n\in \NN$ and the sequence $(g_n)_{n\in \NN}$ converges against $\hat{g}$ with respect
    to the locally convex topology of uniform convergence on all compact subsets of $X$.
  \end{enumerate}
\end{corollary}
Note that in the special case that $\mathcal{I} = \Stetig(X)$, one can drop
in the second point of Proposition~\ref{proposition:Cauchy} and of its
Corollary~\ref{corollary:convergence} the 
condition that there exists a $b\in \mathcal{I}^+$ such 
that $\abs{g_n} \le b$ holds for all $n\in \NN$, because such a $b\in \mathcal{I}^+ = \Stetig(X)^+$
always exists: Let $\hat{g} \in \Stetig(X)$ be the limit of the Cauchy sequence 
$(g_n)_{n\in \NN}$ with respect to the locally convex topology of uniform convergence 
on all compact subsets of $X$, then Lemma~\ref{lemma:convergence}, applied to the sequence
$\big(\abs{g_n-\hat{g}}\big)_{n\in \NN}$, yields a $b' \in \Stetig(X)^+$
fulfilling $\abs{g_n-\hat{g}} \le b'$ for all $n\in \NN$. Hence $\abs{g_n} \le \abs{g_n-\hat{g}} + \abs{\hat{g}} \le b$
for all $n\in \NN$ if one chooses $b \coloneqq b' + \hat{g} \in \Stetig(X)^+$.
So in this special case, the notions of strictly $\mathcal{I}$\=/convergent and 
of strict $\mathcal{I}$\=/Cauchy sequences are equivalent to the corresponding notions of
convergent and Cauchy sequences with respect to the locally convex topology of 
uniform convergence on all compact subsets of $X$. However, in general this is 
not true as was shown in Example~\ref{example:depends}.
\section{The Stone-Weierstraß Theorems}
Even though strict $\mathcal{I}$\=/convergence is stronger than uniform convergence 
on all compact subsets of $X$, variants of the Stone-Weierstraß theorem still hold.
Recall the classical version for Riesz subspaces:
\begin{theoremNN}
  Let $K$ be a compact topological Hausdorff space and $\mathcal{R}$ a Riesz subspace of 
  the Riesz space $\Stetig(K)$. If
  $\mathcal{R}$ is strictly point-separating and nowhere-vanishing, then for every $g\in \Stetig(K)$
  and every $\epsilon \in {]0,\infty[}$ there exists an $r\in \mathcal{R}$ such that
  $\abs[\big]{g(x)-r(x)} \le \epsilon$ for all $x\in K$.
\end{theoremNN}
\begin{proof}
  For convenience of the reader, a proof is given: Note that
  for all $x,y\in K$ with $x\neq y$ there exists an $r_{x,y} \in \mathcal{R}$ with $r_{x,y}(x) = g(x)$
  and $r_{x,y}(y) = g(y)$ because $\mathcal{R}$ is a strictly point-separating 
  linear subspace of $\Stetig(K)$. Given 
  $g\in \Stetig(K)$ and $\epsilon\in {]0,\infty[}$, then the first step is to construct
  for all $x\in K$ an $r_x \in \mathcal{R}$ such that $r_x(x) = g(x)$ and
  $r_x(y) > g(y)-\epsilon$ for all $y\in K$.

  So fix $x\in K$, then there is an $s_x \in \mathcal{R}$ with $s_x(x) = g(x)$ 
  because $\mathcal{R}$ is a nowhere-vanishing linear subspace of $\Stetig(K)$.
  The subset $K_x \coloneqq \set[\big]{y\in K}{s_x(y) \le g(y)-\epsilon}$ of $K$
  is compact and for all $y\in K_x$ there exists an $r_{x,y} \in \mathcal{R}$ 
  like above. The open neighbourhoods $U_y \coloneqq \set[\big]{z\in K}{r_{x,y}(z) > g(z) - \epsilon}$ 
  of $y$ with $y$ ranging over $K_x$ cover the compact $K_x$ and thus there exist
  finitely many points $y_1,\dots,y_M\in K_x$ with $M\in \NN_0$ such that 
  $K_x \subseteq \bigcup_{m=1}^M U_{y_m}$.
  It follows that $r_x \coloneqq s_x \vee \bigvee_{m=1}^M r_{x,y_m} \in \mathcal{R}$ 
  fulfils $r_x(x) = g(x)$ and $r_x(y) > g(y)-\epsilon$ for all $y\in K$.

  Now consider the open neighbourhoods $V_x\coloneqq \set[\big]{y\in K}{r_x(y) < g(y)+\epsilon}$
  of $x\in K$. When $x$ runs over whole $K$, these cover $K$ and so there exist $N\in \NN_0$
  and $x_1,\dots,x_N \in K$ such that $K = \bigcup_{n=1}^N V_x$.
  Then $r \coloneqq \bigwedge_{n=1}^N r_{x_n} \in \mathcal{R}$ fulfils 
  $g(x)-\epsilon < r(x) < g(x)+\epsilon$ for all $x\in K$.
\end{proof}
We can now derive the first version of the generalized Stone-Weierstraß theorem:
\begin{theorem} \label{theorem:convergence}
  Let $\mathcal{I}$ be a Riesz ideal of $\Stetig(X)$ and $\mathcal{R}$ a strictly 
  point-separating and nowhere-vanishing Riesz subspace of $\mathcal{I}$, then the
  $\mathcal{I}$\=/closure of $\mathcal{R}$ is whole $\mathcal{I}$.
\end{theorem}
\begin{proof}
  The first step is to construct an increasing sequence $(e_m)_{m\in \NN}$ in $\mathcal{R}^+$
  with $\sup_{m\in \NN} e_m(x) = \infty$ at all points $x\in X$: For every
  $x\in X$ choose a function $r_x\in\mathcal{R}^+$ with $r_x(x) > 0$.
  Such a function exists because $\mathcal{R}$ is a nowhere-vanishing Riesz subspace of $\mathcal{I}$.
  The open sets $r_x^{-1}\big( ]0,\infty[ \big)$ with $x$ ranging over $X$ cover $X$,
  so there exists a sequence $(x_n)_{n\in \NN}$ in $X$ such that 
  $X = \bigcup_{n\in \NN} r_{x_n}^{-1}\big( ]0,\infty[ \big)$ because $X$ is Lindelöf.
  Now define $e_m \coloneqq m\sum_{n=1}^m r_{x_n} \in \mathcal{R}^+$ for all $m\in \NN$, then the resulting 
  sequence $(e_m)_{m\in \NN}$ is increasing by construction. Moreover, for every $\hat{x}\in X$
  there exists an $M\in \NN$ such that $r_{x_M}(\hat{x}) > 0$, and $e_m(\hat{x}) \ge m \,r_{x_M}(\hat{x})$
  holds for all $m\in \NN$ with $m\ge M$, which thus diverges.

  Now denote the $\mathcal{I}$\=/closure of $\mathcal{R}$ by $\mathcal{R}^\cl$ and
  let $f\in \mathcal{I}^+$ be given, then we have to show that $f\in \mathcal{R}^\cl$.
  Recall that $X$ admits a countable compact exhaustion $(K_n)_{n\in \NN}$ and fix $m\in \NN$.
  For every $n\in \NN$, the classical Stone-Weierstraß theorem, applied to
  the restrictions of $f\wedge e_m$ and of the functions in $\mathcal{R}$ to $K_n$,
  shows that there exists an $r_{m,n}' \in \mathcal{R}$ fulfilling 
  $\abs[\big]{(f\wedge e_m)(x)-r_{m,n}'(x)} \le 1/n$ for all $x\in K_n$. Use this to
  define $r_{m,n} \coloneqq 0 \vee (r_{m,n}\wedge e_m) \in \mathcal{R}^+$, which 
  still fulfils $\abs[\big]{(f\wedge e_m)(x)-r_{m,n}(x)} \le 1/n$ for all $x\in K_n$.
  The resulting sequence $(r_{m,n})_{n\in \NN}$ in $\mathcal{R}^+$ is bounded from 
  above by $e_m \in \mathcal{R}^+ \subseteq \mathcal{I}^+$,
  and converges uniformly on every compact subset of $X$ against $f\wedge e_m$, because
  every compact subset of $X$ is eventually contained in all $K_n$ with $n\ge N$
  for sufficiently large $N\in \NN$. By Corollary~\ref{corollary:convergence},
  $(r_{m,n})_{n\in \NN}$ is strictly $\mathcal{I}$\=/convergent against $f\wedge e_m$,
  so $f\wedge e_m \in \mathcal{R}^\cl$. Finally, the resulting sequence $(f \wedge e_m)_{m\in \NN}$ in $\mathcal{R}^\cl$
  is increasing and has pointwise supremum $f$, hence is strictly $\mathcal{I}$\=/convergent
  against $f$, so $f\in \mathcal{R}^\cl$.
  
  As every element $g\in\mathcal{I}$ can be decomposed as a difference $g=f_+-f_-$
  of two elements $f_+,f_- \in \mathcal{I}^+$, e.g. $f_+ = g\vee 0$ and $f_- = -(g\wedge 0)$,
  and as $\mathcal{R}^\cl$ is a linear subspace of $\mathcal{I}$ by 
  Proposition~\ref{proposition:algebraPermanence}, it follows that $\mathcal{R}^\cl=\mathcal{I}$.
\end{proof}
\begin{example} \label{example:compactlysupported}
  Let $\mathcal{I}$ be a Riesz ideal of $\Stetig(X)$. Then $\Stetig_c(X)$ is strictly 
  point-separating and nowhere-vanishing because $X$ is a Tychonoff space, so the 
  $\mathcal{I}$\=/closure of $\Stetig_c(X)$ is $\mathcal{I}$.
\end{example}
In order to derive the second, algebraic version, one can use the usual
approximation of the square root function on compact intervals by polynomials:
\begin{lemmaNN}
  Let the sequence $( p_n )_{n\in \NN_0}$ of real-valued polynomial functions 
  on $\RR$ be defined recursively by $p_0 \coloneqq 0$ and $p_{n+1} \coloneqq p_n + \frac{1}{2}\big( \id_\RR - (p_n)^2 \big)$
  for all $n\in \NN_0$, where $\id_\RR \colon \RR\to \RR$, $x\mapsto \id_\RR(x)\coloneqq x$.
  Then all these polynomials have vanishing constant term,
  i.e. $p_n(0) = 0$ for all $n\in \NN$, and for every $x\in [0,1]$, the sequence
  $\big( p_n(x) \big)_{n\in \NN}$ in $\RR$ is increasing and has supremum
  $\sup_{n\in \NN} p_n(x) = \sqrt{x}$.
\end{lemmaNN}
\begin{proof}
  For convenience of the reader, a proof is given: If $n=0$, then it is clear that
  $p_0$ has vanishing constant term and that $0 \le p_0(x) \le \sqrt{x}$ for all $x\in [0,1]$.
  Now assume that it has been shown for one $n\in \NN_0$ that $p_n$ has vanishing
  constant term and that $0 \le p_n(x) \le \sqrt{x}$ holds for all $x\in [0,1]$.
  Then $p_{n+1}$ has again vanishing constant term and clearly 
  $p_{n+1}(x) \ge p_n(x) \ge 0$ for all $x\in [0,1]$. But the identity
  $\sqrt{x} - p_{n+1}(x) = \big(\sqrt{x}-p_n(x)\big)\big( 1 - \frac{1}{2}(\sqrt{x}+p_n(x))\big)$
  also implies $p_{n+1}(x) \le \sqrt{x}$ for all $x\in [0,1]$.
  
  It only remains to show that the increasing sequence $\big( p_n(x) \big)_{n\in \NN}$ in 
  $\RR$ has supremum $\sqrt{x}$. The last identity
  already shows that $0 \le \sqrt{x} - p_{n+1}(x) \le \big(\sqrt{x}-p_n(x)\big)\big( 1 - \sqrt{x}/2\big)$,
  so $\sup_{n\in \NN} p_n(x) = \sqrt{x}$ for all $x\in {]0,1]}$. If $x=0$, however,
  then $p_n(0) = 0$ holds for all $n\in \NN$ and again $\sup_{n\in \NN} p_n(0) = \sqrt{0}$.
\end{proof}
\begin{theorem} \label{theorem:convergeAlg}
  Let $\mathcal{I}$ be a Riesz ideal of $\Stetig(X)$ and $\mathcal{B}$ a 
  strictly point-separating and nowhere-vanishing linear subspace of $\mathcal{I} \cap \Stetig_b(X)$
  that is also a subalgebra of $\Stetig_b(X)$, then the $\mathcal{I}$\=/closure of
  $\mathcal{B}$ is whole $\mathcal{I}$.
\end{theorem}
\begin{proof}
  Let 
  $\mathcal{J} \coloneqq \set[\big]{f\in \Stetig(X)}{\exists_{N\in \NN;\,b_1,\dots,b_N\in \mathcal{B}}: \abs{f} \le \sum_{n=1}^N\abs{b_n}}$
  be the Riesz ideal generated by $\mathcal{B}$ in $\Stetig(X)$. Then $\mathcal{J}$ 
  is also a subalgebra of $\Stetig_b(X)$ and $\mathcal{B} \subseteq \mathcal{J} \subseteq \mathcal{I}$.
  Write $\mathcal{B}^\cl_{\mathcal{J}}$ for the $\mathcal{J}$\=/closure of $\mathcal{B}$, then
  the first step is to show that $\mathcal{B}^\cl_{\mathcal{J}}$ is a Riesz subspace of $\mathcal{J}$.
  From Proposition~\ref{proposition:algebraPermanence} it follows that 
  $\mathcal{B}^\cl_{\mathcal{J}}$ is again a subalgebra of $\mathcal{J}$. So 
  $q\circ b \in \mathcal{B}^\cl_{\mathcal{J}}$ for all $b\in \mathcal{B}^\cl_{\mathcal{J}}$ 
  and every polynomial function $q \colon \RR\to\RR$ with vanishing constant term.
  Now let $(p_n)_{n\in \NN}$ be a sequence of polynomials with vanishing constant term 
  such that $\big(p_n(x)\big)_{n\in \NN}$ is an increasing sequence in $\RR$ with supremum $\sqrt{x}$
  for every $x\in [0,1]$.
  Then for every $b \in \mathcal{B}^\cl_{\mathcal{J}}$ with $-\Unit \le b \le \Unit$ the sequence
  $\NN \ni n \mapsto p_n \circ b^2 \in \mathcal{B}^\cl_{\mathcal{J}}$ is increasing and has pointwise
  supremum $\sqrt{\argument} \circ b^2 = \abs{b} \in \mathcal{J}$, so it is 
  $\mathcal{J}$\=/convergent against $\abs{b}$. This shows that $\abs{b} \in \mathcal{B}^\cl_{\mathcal{J}}$
  for all $b \in \mathcal{B}^\cl_{\mathcal{J}}$ with $-\Unit \le b \le \Unit$. As all functions in $\mathcal{B}^\cl_{\mathcal{J}}$
  are uniformly bounded, there exists for every $b\in \mathcal{B}^\cl_{\mathcal{J}}$ a $\lambda \in {]0,\infty[}$
  such that $- \Unit \le \lambda b \le \Unit$ and thus $\abs{b} = \lambda^{-1} \abs{\lambda b} \in \mathcal{B}^\cl_{\mathcal{J}}$.
  We conclude that $\mathcal{B}^\cl_{\mathcal{J}}$ is a Riesz subspace of $\mathcal{J}$, and it is
  again strictly point-separating and nowhere-vanishing because it contains $\mathcal{B}$. 
  
  Now let $\mathcal{B}^\cl_{\mathcal{I}}$ be the $\mathcal{I}$\=/closure of $\mathcal{B}$,
  then $\mathcal{B}^\cl_{\mathcal{I}} \cap \mathcal{J}$ is $\mathcal{J}$\=/closed, because
  every strictly $\mathcal{J}$\=/convergent sequence $(g_n)_{n\in \NN}$ in 
  $\mathcal{B}^\cl_{\mathcal{I}} \cap \mathcal{J}$ with strict $\mathcal{J}$\=/limit 
  $\hat{g} \in \mathcal{J}$ is also strictly $\mathcal{I}$\=/convergent against the 
  same strict $\mathcal{I}$\=/limit $\hat{g}$ because $\mathcal{J}\subseteq \mathcal{I}$.
  Hence $\hat{g} \in \mathcal{B}^\cl_{\mathcal{I}}$
  and thus even $\hat{g} \in \mathcal{B}^\cl_{\mathcal{I}} \cap \mathcal{J}$. 
  Consequently, $\mathcal{B}^\cl_{\mathcal{J}} \subseteq \mathcal{B}^\cl_{\mathcal{I}} \cap \mathcal{J}
  \subseteq \mathcal{B}^\cl_{\mathcal{I}}$ holds, so the $\mathcal{I}$\=/closure
  $\mathcal{B}^\cl_{\mathcal{I}}$ of $\mathcal{B}$ contains the 
  strictly point-separating and nowhere-vanishing Riesz subspace $\mathcal{B}^\cl_{\mathcal{J}}$
  of $\mathcal{I}$. By Theorem~\ref{theorem:convergence}, this means that 
  $\mathcal{B}^\cl_{\mathcal{I}}=\mathcal{I}$.
\end{proof}
\begin{example}
  Assume that $X$ is a smooth manifold (which is locally compact Hausdorff and second
  countable, hence indeed a Lindelöf space) and let $\mathcal{I}$ be a Riesz ideal of 
  $\Stetig(X)$. Then $\mathcal{I}$ contains the space $\Stetig^\infty_c(X)$ of all
  smooth real-valued functions on $X$ with compact support because 
  $\Stetig^\infty_c(X) \subseteq \Stetig_c(X) \subseteq \mathcal{I}$. Moreover,
  $\Stetig^\infty_c(X)$ is strictly point-separating and nowhere-vanishing
  because for every $x\in X$ and every open neighbourhood $U$ of $x$
  with compact closure one can (using local coordinates around $x$) construct an 
  $f_x\in \Stetig^\infty_c(X)$ with $f_x(x)\neq 0$ and $f_x(y) = 0$ for all $y\in X\backslash U$.
  So the $\mathcal{I}$\=/closure of $\Stetig^\infty_c(X)$ is whole $\mathcal{I}$.
\end{example}
Important consequences of these theorems are sufficient conditions under which two positive 
linear maps on a Riesz ideal of $\Stetig(X)$ coincide. These make use of the following
lemma:
\begin{lemma} \label{lemma:PhiPsi}
  Let $\mathcal{I}$ be a Riesz ideal of $\Stetig(X)$ and $V$ an ordered vector space.
  If $\Phi,\Psi \colon \mathcal{I}\to V$ are two positive linear maps and strictly 
  $\mathcal{I}$\=/continuous, then the subset $\set[\big]{g\in \mathcal{I}}{\Phi(g)=\Psi(g)}$
  of $\mathcal{I}$ on which $\Phi$ and $\Psi$ coincide is $\mathcal{I}$\=/closed.
\end{lemma}
\begin{proof}
  Let $(g_n)_{n\in \NN}$ be a sequence in $\set[\big]{g\in \mathcal{I}}{\Phi(g)=\Psi(g)}$
  which is strictly $\mathcal{I}$\=/convergent against a strict $\mathcal{I}$\=/limit 
  $\hat{g} \in \mathcal{I}$. Then there exists a decreasing sequence $(f_k)_{k\in \NN}$ in $\mathcal{I}$
  with pointwise infimum $0$ and such that for all $k\in \NN$ there is an $N\in \NN$ so that
  $\abs{\hat{g}-g_n} \le f_k$ for all $n\in \NN$ with $n\ge N$.
  As 
  \begin{equation*}
    \Phi(\hat{g})-\Psi(\hat{g})
    = 
    \Phi(\hat{g}-g_n)+\Psi(g_n-\hat{g})
    \le
    \Phi\big(\abs{\hat{g}-g_n}\big)+\Psi\big(\abs{g_n-\hat{g}}\big)
  \end{equation*}
  holds for all $n\in \NN$, it follows that the estimate
  $\Phi(\hat{g})-\Psi(\hat{g}) \le \Phi\big(\abs{\hat{g}-g_n} \big) + \Psi\big(\abs{\hat{g}-g_n}\big)
  \le \Phi(f_k) + \Psi(f_\ell)$ holds for all $k,\ell \in \NN$ if $n\in \NN$ is 
  chosen sufficiently large ($n$ may depend on $k$ and $\ell$). As $\Phi$ and $\Psi$
  are both strictly $\mathcal{I}$\=/continuous by assumption, 
  this implies $\Phi(\hat{g})-\Psi(\hat{g}) \le \inf_{k\in \NN}\Phi(f_k) + \inf_{\ell \in \NN} \Psi(f_\ell) = 0$,
  i.e. $\Phi(\hat{g}) \le \Psi(\hat{g})$. Exchanging $\Phi$ and $\Psi$ shows $\Phi(\hat{g}) \ge \Psi(\hat{g})$,
  so $\Phi(\hat{g}) = \Psi(\hat{g})$.
\end{proof}
An immediate consequence of Theorems~\ref{theorem:convergence}
and \ref{theorem:convergeAlg} and this Lemma~\ref{lemma:PhiPsi} is:
\begin{proposition} \label{proposition:coincidence}
  Let $\mathcal{I}$ be a Riesz ideal of $\Stetig(X)$, let $V$ be an ordered vector
  space and $\Phi,\Psi\colon \mathcal{I} \to V$ two positive linear maps that are
  strictly $\mathcal{I}$\=/continuous. If $\Phi$ and $\Psi$ coincide on a strictly 
  point-separating and nowhere-vanishing Riesz subspace $\mathcal{R}$ of $\mathcal{I}$,
  or on a strictly point-separating and nowhere-vanishing linear subspace $\mathcal{B}$
  of $\mathcal{I} \cap \Stetig_b(X)$ that is also a subalgebra of $\Stetig_b(X)$, 
  then $\Phi = \Psi$.
\end{proposition}
This yields a condition for Riesz morphisms to coincide,
as every Riesz morphism is especially a positive linear map and as the set on which
two Riesz morphisms coincide is a Riesz subspace:
\begin{corollary} \label{corollary:coincidenceRiesz}
  Let $\mathcal{I}$ be a Riesz ideal of $\Stetig(X)$ and $\mathcal{R}$ a Riesz space.
  If $\Phi,\Psi\colon \mathcal{I} \to \mathcal{R}$ are two morphisms of Riesz spaces that
  are strictly $\mathcal{I}$\=/continuous and that coincide on a strictly point-separating
  and nowhere-vanishing subset $S$ of $\mathcal{I}$, then $\Phi = \Psi$.
\end{corollary}
The algebraic version of Proposition~\ref{proposition:coincidence} is
not so trivial to apply, because it refers to a subalgebra of \emph{uniformly bounded}
functions. But sometimes this can be guaranteed:
\begin{corollary} \label{corollary:coincidenceAlg}
  Let $\mathcal{I}$ be a Riesz ideal and unital subalgebra of $\Stetig(X)$ and $\mathcal{A}$ 
  an ordered unital algebra. If $\Phi,\Psi\colon \mathcal{I} \to \mathcal{A}$ are two 
  unital morphisms of algebras that are strictly $\mathcal{I}$\=/continuous and that
  coincide on a point-separating subset $S$ of $\mathcal{I}$, then $\Phi = \Psi$.
\end{corollary}
\begin{proof}
  First note that $\Phi$ and $\Psi$ are automatically positive, so the assumption
  of strict $\mathcal{I}$\=/continuity of $\Phi$ and $\Psi$ makes sense: Indeed, every $g\in \Stetig(X)^+$
  is a square of its square root $\sqrt{g} \in \Stetig(X)^+$, and if $g\in \mathcal{I}^+$,
  then $0 \le (\Unit-\sqrt{g})^2 = \Unit-2\sqrt{g} + g$ shows that $\sqrt{g} \le (\Unit+g)/2$,
  so $\sqrt{g} \in \mathcal{I}$ as well. Because of this, $\Phi$ and $\Psi$ map $g = \sqrt{g}^{2}$
  to a square, which is positive in $\mathcal{A}$.
  
  From $\Unit \in \mathcal{I}$ it follows that $\Stetig_b(X)\subseteq \mathcal{I}$.
  The subset $\mathcal{B} \coloneqq \set[\big]{g\in \Stetig_b(X)}{\Phi(g)=\Psi(g)}$ of $\mathcal{I}$ 
  is a unital subalgebra of $\Stetig_b(X)$ because $\Phi$ and $\Psi$ are unital morphisms of algebras.
  In order to apply Proposition~\ref{proposition:coincidence}, it only remains to show that 
  $\mathcal{B}$ is point-separating (hence strictly point-separating and nowhere-vanishing because
  $\Unit\in\mathcal{B}$):
  Given $x,y\in X$ with $x\neq y$, then by assumption there exists
  an $s\in S$ such that $s(x) \neq s(y)$. Use this to define $s_+ \coloneqq (s+\Unit)^2 + \Unit$
  and $s_- \coloneqq (s-\Unit)^2 + \Unit$, then $4s = s_+ - s_-$. So at least one of
  $s_+$ and $s_-$, which will be denoted by $s_\pm$,
  fulfils $s_\pm(x)\neq s_\pm(y)$. From $\Phi(s) = \Psi(s)$ it follows that $\Phi(s_\pm) = \Psi(s_\pm)$.
  As $s_\pm \ge \Unit$, the pointwise multiplicative inverse $s_\pm^{-1} \in \Stetig_b(X)$
  exists and fulfils $s_\pm^{-1}(x) \neq s_\pm^{-1}(y)$, but
  $\Phi\big( s_\pm^{-1} \big)
    =
    \Phi\big( s_\pm^{-1} \big) \Psi\big( s_\pm \big) \Psi\big( s_\pm^{-1} \big)
    =
    \Phi\big( s_\pm^{-1} \big) \Phi\big( s_\pm \big) \Psi\big( s_\pm^{-1} \big)
    =
    \Psi\big( s_\pm^{-1} \big)$
  also holds, so $s_\pm^{-1} \in \mathcal{B}$.
\end{proof}
\section{Automatic Continuity}
The notion of strictly $\mathcal{I}$\=/convergent sequences introduced in Definition~\ref{definition:convergence},
as well as its equivalent description given by Corollary~\ref{corollary:convergence},
clearly makes use of the underlying space $X$. However, in some important special cases,
there is a third characterization of such sequences as the ``relatively uniformly
convergent'' ones. This notion of convergence is completely order-theoretic and well-known
in the theory of ordered vector spaces.

Recall that a function $p \in \Stetig(X)$
is \neu{proper} if the preimage $p^{-1}\big([a,b]\big)$ is compact for all $a,b\in \RR$ with $a\le b$.
There is a generalization of Dini's Theorem:
\begin{lemma} \label{lemma:Dini2}
  Let $\mathcal{A}$ be a unital subalgebra of $\Stetig(X)$ that contains a proper function $p \in \mathcal{A}^+$.
  If $(f_k)_{k\in \NN}$ is a decreasing sequence in $\mathcal{A}$ with pointwise infimum $0$,
  then there exists a function $h \in \mathcal{A}^+$ such that for all $\epsilon \in {]0,\infty[}$ 
  there is a $k\in \NN$ fulfilling $f_k \le \epsilon h$.
\end{lemma}
\begin{proof}
  One can choose $h\coloneqq \Unit + p f_1\in \mathcal{A}^+$:
  Given $\epsilon \in {]0,\infty[}$, construct the compact $K\coloneqq p^{-1}\big( [0,1/\epsilon] \big)$.
  By Dini's theorem, there exists a $k\in \NN$ such that $f_k(x) \le \epsilon$ holds
  for all $x\in K$, thus also $f_k(x) \le \epsilon h(x)$ for all $x\in K$ as $h \ge \Unit$.
  If $x \in X\backslash K$, however, then $p(x) > 1/\epsilon$ and thus 
  $f_k(x) \le f_1(x) \le \epsilon h(x)$ holds as well.
\end{proof}
\begin{proposition}
  Let $\mathcal{I}$ be a Riesz ideal and unital subalgebra of $\Stetig(X)$ that contains
  a proper function $p\in \mathcal{I}^+$. Then for every
  sequence $(g_n)_{n\in \NN}$ in $\mathcal{I}$ and all $\hat{g} \in \mathcal{I}$,
  the following is equivalent:
  \begin{enumerate}
    \item The sequence $(g_n)_{n\in \NN}$ is strictly $\mathcal{I}$\=/convergent 
    with strict $\mathcal{I}$\=/limit $\hat{g}$.
    \item There is a function $h\in\mathcal{I}^+$ with the following property:
    For every $\epsilon \in {]0,\infty[}$ there exists an $N\in \NN$ such that
    $\abs{\hat{g}-g_n} \le \epsilon h$ for all $n\in \NN$ with $n\ge N$.
  \end{enumerate}
\end{proposition}
\begin{proof}
  It is clear that the second point implies the first one.
  Conversely, assume that there
  exists a decreasing sequence $(f_k)_{k\in \NN}$ with pointwise infimum $0$ and such
  that for all $k\in \NN$ there is an $N\in \NN$ for which 
  $\abs{\hat{g} - g_n} \le f_k$ is fulfilled for all $n\in \NN$ with $n\ge N$.
  Then by the previous Lemma~\ref{lemma:Dini2}, there also exists a function $h\in\mathcal{I}^+$ 
  such that for all $\epsilon \in {]0,\infty[}$ there exists a $k\in \NN$
  fulfilling $f_k \le \epsilon h$. This function thus has the property
  required in the second point.
\end{proof}
Applying this characterization of strictly $\mathcal{I}$\=/convergent
sequences (hence of $\mathcal{I}$\=/closed sets) to Theorem~\ref{theorem:convergeAlg}
yields essentially the Stone-Weierstraß-type theorem from
\cite[Thm.~4.1]{jurzak:dominatedConvergenceAndStoneWeierstrass}.
With respect to positive linear maps one can even show the following:
\begin{theorem} \label{theorem:autocont}
  Let $\mathcal{I}$ be a Riesz ideal and unital subalgebra of $\Stetig(X)$ that contains
  a proper function $p\in \mathcal{I}^+$. Then every positive linear map
  $\Phi \colon \mathcal{I} \to V$ into an Archimedean ordered vector space $V$
  is strictly $\mathcal{I}$\=/continuous.
\end{theorem}
\begin{proof}
  Given an Archimedean ordered vector space $V$, a positive linear map $\Phi$ and a
  decreasing sequence $(f_k)_{k\in\NN}$ in $\mathcal{I}$ with pointwise infimum $0$, 
  then $0$ is of course a lower bound in $V$ of all $\Phi(f_k)$ with $k\in \NN$, and 
  it is even the greatest lower bound: 
  By Lemma~\ref{lemma:Dini2}, there exists an $h\in \mathcal{I}^+$ such that 
  for all $\epsilon \in {]0,\infty[}$ there is a $k\in \NN$ fulfilling $f_k \le \epsilon h$,
  hence $\Phi(f_k) \le \epsilon \Phi(h)$. So if $v\in V$ is a lower bound of all $\Phi(f_k)$
  with $k\in \NN$, then also $v\le \epsilon \Phi(h)$ for all $\epsilon \in {]0,\infty[}$
  and even $v\le 0$ because $V$ is Archimedean.
\end{proof}
It is clear that this allows to drop the continuity assumption in Proposition~\ref{proposition:coincidence}
and its Corollaries~\ref{corollary:coincidenceRiesz} and \ref{corollary:coincidenceAlg}
under certain circumstances. Note also the application to Daniell integrals, see Example~\ref{example:integrals}.
\section{The Moment Problem}
With respect to the determinacy of moment problems one would like to understand whether two positive
linear functions from a Riesz ideal and unital subalgebra $\mathcal{I}$ of $\Stetig(X)$ to $\RR$, 
that are strictly $\mathcal{I}$\=/continuous and coincide on a point-separating 
unital subalgebra $\mathcal{A}$ of $\mathcal{I}$, coincide on whole $\mathcal{I}$.
By Proposition~\ref{proposition:coincidence}, this is the case if the two 
functions coincide even on a point-separating unital subalgebra $\mathcal{B}$ of $\Stetig_b(X)$.
In order to guarantee this, an additional assumption of (essential) self-adjointness 
will be necessary.

Before diving into the details, note that it is well-known that there exist
indeterminate moment sequences. For example, let $X = \RR$ and let $\mathcal{I}$ be
the Riesz ideal and unital subalgebra of $\Stetig(\RR)$ given by all polynomially
bounded continuous functions, i.e. the set of all $f\in \Stetig(\RR)$ for which 
there exists a polynomial function $p \in \Stetig(\RR)$ such that $\abs{f} \le p$.
Moreover, let $\mathcal{A}$ be the point-separating unital subalgebra of $\mathcal{I}$
of all polynomial functions. Then there exist two positive Borel measures 
$\mu, \mu'$ on $\RR$ for which $\int_X p \D\mu = \int_X p \D\mu'$ holds for all
$p\in\mathcal{A}$, but not $\int_X f \D\mu = \int_X f \D\mu'$ for all $f\in\mathcal{I}$,
see e.g. the classical example by Stieltjes, \cite[Expl.~4.22]{schmuedgen:TheMomentProblem}.
So the $\mathcal{I}$\=/closure of $\mathcal{A}$ cannot be whole $\mathcal{I}$.
This shows that the assumption in Theorem~\ref{theorem:convergeAlg} that
$\mathcal{B}$ consists of \emph{uniformly bounded} functions is crucial.
\subsection{Intermezzo: \texorpdfstring{$^*$\=/Algebras}{*-Algebras}} \label{sec:staralg}
Before we are able to apply the previous results, some general considerations 
will be necessary:

In the following, let $\mathcal{I}_\CC$ be a
\neu{unital $^*$-algebra}, i.e. a unital associative (but not necessarily commutative) complex 
algebra with an antilinear involution $\argument^*\colon \mathcal{I}_\CC\to \mathcal{I}_\CC$ 
that fulfils $(fg)^* = g^*f^*$ for all $f,g\in\mathcal{I}_\CC$. For example, $\CC$ with the
complex conjugation $\cc{\argument}$, or the algebra of complex-valued continuous functions
on $X$ with the pointwise complex conjugation as $^*$\=/involution are (commutative) unital
$^*$\=/algebras. The unit of $\mathcal{I}_\CC$ will again be denoted by $\Unit$,
and one can check that necessarily $\Unit^* = \Unit$ holds.
An element $f\in\mathcal{I}_\CC$ that fulfils $f^*=f$ is called \neu{Hermitian}.
Moreover, let $\Phi_\CC\colon \mathcal{I}_\CC \to \CC$
be a \neu{positive Hermitian linear map}, i.e. $\Phi_\CC(f^*) = \cc{\Phi_\CC(f)}$ and 
$\Phi_\CC(f^*f) \ge 0$ for all $f\in\mathcal{I}_\CC$. Then 
$\skal{\argument}{\argument}_\Phi \colon \mathcal{I}_\CC\times \mathcal{I}_\CC \to \CC$,
\begin{equation*}
  (f,g) \mapsto \skal{f}{g}_\Phi \coloneqq \Phi_\CC(f^* g)
\end{equation*}
is a positive Hermitian sesquilinear form on $\mathcal{I}_\CC$, i.e. $\cc{\skal{f}{g}}_\Phi = \skal{g}{f}_\Phi$
and $\skal{f}{f}_\Phi\ge 0$
hold for all $f,g\in\mathcal{I}_\CC$, and 
$\seminorm{\Phi}{\argument} \colon \mathcal{I}_\CC \to {[0,\infty[}$,
\begin{equation*}
  f \mapsto \seminorm{\Phi}{f} \coloneqq \skal{f}{f}^{1/2}_\Phi
\end{equation*}
is the corresponding Hilbert seminorm. Note that $\skal{f}{gh}_\Phi = \skal{g^*f}{h}_\Phi$
for all $f,g,h\in\mathcal{I}_\CC$.
Even though $\skal{\argument}{\argument}_\Phi$ 
is not an inner product in general, i.e. $\skal{f}{f}_\Phi=0$ with $f\in\mathcal{I}_\CC$
does not necessarily imply $f=0$, and even though the topology
defined by the seminorm $\seminorm{\Phi}{\argument}$ is not Hausdorff in general,
notions of closure, continuity, Cauchy and convergent sequences with respect to $\seminorm{\Phi}{\argument}$
are well-defined on $\mathcal{I}_\CC$.
However, limits of $\seminorm{\Phi}{\argument}$\=/convergent sequences are not necessarily uniquely determined.
The \neu{Cauchy Schwarz inequality}
\begin{equation*}
  \abs[\big]{\skal{f}{g}_\Phi} \le \seminorm{\Phi}{f}\seminorm{\Phi}{g}
\end{equation*}
also holds for all $f,g\in \mathcal{I}_\CC$. While the left multiplication on 
$\mathcal{I}_\CC$ with an element of $\mathcal{I}_\CC$ is not 
$\seminorm{\Phi}{\argument}$\=/continuous in general,
there is an important special case: If $f = f^*$ and $f+\I\lambda \Unit$ with 
$\lambda \in \{-1,+1\}$ is invertible in $\mathcal{I}_\CC$, then the left
multiplication with $(f+\I\lambda\Unit)^{-1}$ is indeed $\seminorm{\Phi}{\argument}$\=/continuous.
More precisely, the estimate
\begin{equation*}
  \seminorm[\big]{\Phi}{(f+\I\lambda\Unit)^{-1} g} 
  \le
  \Phi_\CC\Big( g^* (f-\I\lambda\Unit)^{-1} \big(f^2+\Unit\big) (f+\I\lambda\Unit)^{-1} g \Big)^\frac{1}{2}
  =
  \Phi_\CC (g^* g)^\frac{1}{2}
  =
  \seminorm{\Phi}{g}
\end{equation*}
holds for all $g\in\mathcal{I}_\CC$. Nevertheless, the left multiplication with
arbitrary elements of $\mathcal{I}_\CC$ still fulfils a weaker condition:
\begin{lemma} \label{lemma:closable}
  Let $\mathcal{I}_\CC$ be a unital $^*$\=/algebra, $f\in \mathcal{I}_\CC$ and $\Phi_\CC\colon \mathcal{I}_\CC \to \CC$
  a positive Hermitian linear map.
  If a sequence $(g_n)_{n\in \NN}$ in $\mathcal{I}_\CC$ is $\seminorm{\Phi}{\argument}$\=/convergent
  against some $\hat{g} \in \mathcal{I}_\CC$ and such that the sequence $(fg_n)_{n\in\NN}$
  is a $\seminorm{\Phi}{\argument}$\=/Cauchy sequence, then $(fg_n)_{n\in \NN}$ is 
  also $\seminorm{\Phi}{\argument}$\=/convergent against $f\hat{g}$.
\end{lemma}
\begin{proof}
  This is a consequence of the estimate
  \begin{align*}
    \seminorm{\Phi}{fg_n-f\hat{g}}^2 
    &=
    \skal{fg_n-fg_N}{fg_n-f\hat{g}}_\Phi + \skal{f^*fg_N-f^*f \hat{g}}{g_n-\hat{g}}_\Phi
    \\
    &\le
    \seminorm{\Phi}{fg_n-fg_N} \seminorm{\Phi}{fg_n-f\hat{g}} + \seminorm{\Phi}{f^*f g_N-f^*f \hat{g}} \seminorm{\Phi}{g_n-\hat{g}}
  \end{align*}
  for $n,N\in \NN$:
  Given $\epsilon \in {]0,\infty[}$. As $(fg_n)_{n\in \NN}$ is a Cauchy sequence,
  there exists an $N\in \NN$ such that $\seminorm{\Phi}{fg_n-fg_N} \le 2\epsilon / 3$ holds 
  for all $n\in \NN$ with $n\ge N$. But as the sequence $(g_n)_{n\in \NN}$ is $\seminorm{\Phi}{\argument}$\=/convergent
  against $\hat{g}$, there also exists an $N' \in \NN$ such that 
  $\seminorm{\Phi}{f^*f g_N-f^*f \hat{g}} \seminorm{\Phi}{g_n-\hat{g}} \le \epsilon^2 / 3 $ holds for 
  the above $N$ and all $n\in \NN$ with $n\ge N'$. Consequently, the quadratic inequality
  $3\seminorm{\Phi}{fg_n-f\hat{g}}^2 \le 2\epsilon \seminorm{\Phi}{fg_n-f\hat{g}} + \epsilon^2$
  holds for all $n\in \NN$ with $n\ge \max\{N,N'\}$, which implies $\seminorm{\Phi}{fg_n-f\hat{g}} \le \epsilon$.
\end{proof}
Next let $\mathcal{A}_\CC$ be a \neu{unital $^*$\=/subalgebra} of $\mathcal{I}_\CC$, 
i.e. a unital (complex) subalgebra such that $a^*\in \mathcal{A}_\CC$ holds for all 
$a\in\mathcal{A}_\CC$. Write $\mathcal{A}^\cl_\CC$ for its $\seminorm{\Phi}{\argument}$\=/closure.
Note that $\mathcal{A}^\cl_\CC$ is the set of all $\seminorm{\Phi}{\argument}$\=/limits of $\seminorm{\Phi}{\argument}$\=/convergent
sequences in $\mathcal{A}_\CC$, because -- contrary to the notions from 
Definition~\ref{definition:convergence} -- the notions of $\seminorm{\Phi}{\argument}$\=/convergent
sequences and $\seminorm{\Phi}{\argument}$\=/closure used here come from a first countable topology on $\mathcal{I}_\CC$.
Similarly, for all $a\in \mathcal{A}_\CC$ with $a=a^*$, the \neu{$a$-$\seminorm{\Phi}{\argument}$-closure} 
$\mathcal{A}^\acl_\CC$ of $\mathcal{A}_\CC$ is defined as the closure
with respect to the seminorm 
$\mathcal{I}_\CC \ni g \mapsto \big( \seminorm{\Phi}{ag}^2 + \seminorm{\Phi}{g}^2 \big)^{1/2} = 
\seminorm{\Phi}{(a+\I\lambda \Unit)g} \in {[0,\infty[}$ with $\lambda \in \{-1,+1\}$.
Then $\mathcal{A}^\acl_\CC$ is the
set of all $\seminorm{\Phi}{\argument}$\=/limits 
$\hat{b} \in \mathcal{I}_\CC$ of $\seminorm{\Phi}{\argument}$\=/convergent sequences 
$(b_n)_{n\in \NN}$ in $\mathcal{A}_\CC$ for which additionally $(ab_n)_{n\in \NN}$ 
is $\seminorm{\Phi}{\argument}$\=/convergent against $a\hat{b}$. 
Clearly $\mathcal{A}^\acl_\CC\subseteq \mathcal{A}^\cl_\CC$.

In the following we will also need some operator-theoretic notions. For these to be
applicable, it is necessary to divide out the kernel of $\seminorm{\Phi}{\argument}$
in order to get a true (pre-)Hilbert space. We will only work with the unital $^*$\=/subalgebra
$\mathcal{A}_\CC$ of $\mathcal{I}_\CC$: 
The \neu{Gel'fand ideal} with respect to $\Phi_\CC$ is defined as
$\mathcal{G}_\Phi \coloneqq \set[\big]{a\in\mathcal{A}_\CC}{\seminorm{\Phi}{a}=0} = 
\set[\big]{a\in\mathcal{A}_\CC}{\forall_{b\in\mathcal{A}_\CC}: \skal{b}{a}_\Phi = 0}$,
and we write $[\argument] \colon \mathcal{A}_\CC \to \mathcal{A}_\CC\big/\mathcal{G}_\Phi$
for the canonical projection onto the quotient. Then $\mathcal{G}_\Phi$
is a left ideal of $\mathcal{A}_\CC$ and thus all left multiplications with
elements $a\in \mathcal{A}_\CC$ descend to well-defined linear endomorphisms
$\mult{a} \colon \mathcal{A}_\CC \big/ \mathcal{G}_\Phi \to \mathcal{A}_\CC \big/ \mathcal{G}_\Phi$,
\begin{equation*}
  [b] 
  \mapsto 
  \mult{a}[b] 
  \coloneqq 
  [ab]\,.
\end{equation*}
The positive sesquilinear form $\skal{\argument}{\argument}_\Phi$ and its induced 
seminorm $\seminorm{\Phi}{\argument}$ (or rather, their restrictions to $\mathcal{A}_\CC$)
also descend to this quotient, where 
$\skal{\argument}{\argument}_\Phi$ is an inner product and $\seminorm{\Phi}{\argument}$ is a norm. 
So $\mathcal{A}_\CC \big/ \mathcal{G}_\Phi$
with this inner product is a pre-Hilbert space and all endomorphisms $\mult{a}$ with $a\in\mathcal{A}_\CC$ 
of $\mathcal{A}_\CC \big/ \mathcal{G}_\Phi$ are adjointable (in the algebraic sense)
with $(\mult{a})^* = \mult{a^*}$. Let $\Hilb_\Phi$ be the completion of 
$\mathcal{A}_\CC \big/ \mathcal{G}_\Phi$ to a Hilbert space. With some 
abuse of notation, we can identify $\mathcal{A}_\CC \big/ \mathcal{G}_\Phi$ 
with a $\seminorm{\Phi}{\argument}$\=/dense linear subspace of $\Hilb_\Phi$ so that the $\mult{a}$ with $a\in\mathcal{A}_\CC$ 
form a $^*$\=/algebra of (not necessarily bounded) operators on $\Hilb_\Phi$. 
This is the well-known GNS-representation of the $^*$\=/algebra $\mathcal{A}_\CC$
with respect to the positive Hermitian linear map $\Phi_\CC$. Note that 
the canonical projection $[\argument] \colon \mathcal{A}_\CC \to \mathcal{A}_\CC\big/\mathcal{G}_\Phi$
is a $\seminorm{\Phi}{\argument}$\=/continuous linear map and thus extends to a $\seminorm{\Phi}{\argument}$\=/continuous
linear map from $\mathcal{A}_\CC^\cl$ to $\Hilb_\Phi$, which will also be denoted by $[\argument]$.
But one has to be careful here: While $\skal{[a]}{[b]}_\Phi = \Phi_\CC(a^*b)$ holds
even for all $a,b\in \mathcal{A}_\CC^\cl$ (and not only for $a,b\in \mathcal{A}_\CC$) 
due to the continuity of both sides as sesquilinear maps from $\mathcal{A}_\CC^\cl \times \mathcal{A}_\CC^\cl$
to $\CC$, the identity $\mult{a}[b] = [ab]$ does not even make any sense for general
$a\in\mathcal{A}_\CC$ and $b\in\mathcal{A}_\CC^\cl$ because the left-hand side is
not defined. As $\mult{a}$ need not be continuous, it is not so easy to
extend $\mult{a}$ in a suitable way.

If $a\in \mathcal{A}_\CC$ is Hermitian, then $\mult{a}$ is a symmetric operator 
on $\Hilb_\Phi$, i.e. $\skal{[b]}{\mult{a}[c]}_\Phi = \skal{\mult{a}[b]}{[c]}_\Phi$
for all $[b],[c] \in \mathcal{A}_\CC \big/ \mathcal{G}_\Phi$. The case that $\mult{a}$
is even \neu{essentially self-adjoint} is especially interesting. By definition,
this means that the domain of the closure of $\mult{a}$ coincides with the domain
of the (operator theoretic) adjoint of $\mult{a}$. While the domain of the
closure of $\mult{a}$ is the set of all limits in $\Hilb_\Phi$ of $\seminorm{\Phi}{\argument}$\=/Cauchy sequences 
$\big( [b_n]\big)_{n\in \NN}$ in $\mathcal{A}_\CC \big/ \mathcal{G}_\Phi$ for which
$\big( \mult{a}[b_n]\big)_{n\in \NN}$ is also a $\seminorm{\Phi}{\argument}$\=/Cauchy sequence, the domain of
the operator-theoretic adjoint of $\mult{a}$ is the set of all $\psi \in \Hilb_\Phi$
for which the linear map $\mathcal{A}_\CC \big/ \mathcal{G}_\Phi \ni [b] \mapsto \skal{\psi}{\mult{a}[b]}_\Phi \in \CC$
is $\seminorm{\Phi}{\argument}$\=/continuous. An alternative characterization of essential self-adjointness
can be given with the help of von Neumann's formulas: The symmetric operator $\mult{a}$ 
is essentially self-adjoint if and only if the image of 
$\mathcal{A}_\CC \big/ \mathcal{G}_\Phi$ under $\mult{a+\I\lambda\Unit}$ is dense 
in $\mathcal{A}_\CC \big/ \mathcal{G}_\Phi$ for both $\lambda \in \{-1,+1\}$.

The above notions and considerations are more or less standard.
For the discussion of the determinacy of moment problems, we need only one special result:
\begin{proposition} \label{proposition:staralg}
  Let $\mathcal{I}_\CC$ be a unital $^*$\=/algebra, $\Phi_\CC\colon \mathcal{I}_\CC \to \CC$
  a positive Hermitian linear map and $\mathcal{A}_\CC$ a unital $^*$\=/subalgebra 
  of $\mathcal{I}_\CC$. Moreover, let $a\in\mathcal{A}_\CC$ with $a=a^*$ be given
  and such that $a+\I\lambda\Unit$ is invertible in $\mathcal{I}_\CC$ for both $\lambda \in \{-1,+1\}$.
  Using the notation introduced above, the following is equivalent:
  \begin{enumerate}
    \item $\mathcal{A}_\CC^{\acl} = \mathcal{A}_\CC^\cl$ and 
    $(a+\I\lambda\Unit)^{-1} b \in \mathcal{A}_\CC^\cl$ for all $b\in\mathcal{A}_\CC^\cl$
    and both $\lambda \in \{-1,+1\}$.
    \label{item:staralg:1}
    \item $(a+\I\lambda\Unit)^{-1} b \in \mathcal{A}_\CC^\acl$ for all $b\in\mathcal{A}_\CC^\acl$
    and both $\lambda \in \{-1,+1\}$.
    \label{item:staralg:2}
    \item The symmetric operator $\mult{a}$ in the GNS-representation of $\mathcal{A}_\CC$
    with respect to $\Phi_\CC$ is essentially self-adjoint.
    \label{item:staralg:3}
  \end{enumerate}
\end{proposition}
\begin{proof}
  \refitem{item:staralg:1} clearly implies \refitem{item:staralg:2}, and 
  \refitem{item:staralg:2} implies \refitem{item:staralg:3} because if 
  \refitem{item:staralg:2} holds, then the image of $\mathcal{A}_\CC \big/ \mathcal{G}_\Phi$
  under $\mult{a+\I\lambda\Unit}$ is dense in $\mathcal{A}_\CC \big/ \mathcal{G}_\Phi$ for both $\lambda \in \{-1,+1\}$:
  Given $c\in\mathcal{A}_\CC$, then $(a+\I\lambda)^{-1}c \in \mathcal{A}_\CC^{\acl}$,
  so there exists a sequence $(b_n)_{n\in \NN}$ in $\mathcal{A}_\CC$ that is 
  $\seminorm{\Phi}{\argument}$\=/convergent against $(a+\I\lambda)^{-1}c$ and 
  for which $(ab_n)_{n\in \NN}$ is $\seminorm{\Phi}{\argument}$\=/convergent against
  $a(a+\I\lambda)^{-1}c$. It follows that $\big((a+\I\lambda\Unit) b_n\big)_{n\in \NN}$
  is $\seminorm{\Phi}{\argument}$\=/convergent against $c$, so the sequence $\big(\mult{a+\I\lambda\Unit}[b_n]\big)_{n\in \NN}$
  in the image of $\mathcal{A}_\CC \big/ \mathcal{G}_\Phi$ under $\mult{a+\I\lambda\Unit}$
  is $\seminorm{\Phi}{\argument}$\=/convergent against $[c]$.

  It only remains to show that \refitem{item:staralg:3} implies \refitem{item:staralg:1},
  so assume that $\mult{a}$ is essentially self-adjoint. Given $c\in \mathcal{A}^\cl_\CC$,
  then
  $\abs[\big]{\skal{[c]}{[ab]}_\Phi} = \abs[\big]{\Phi_\CC(c^*ab)} = \abs[\big]{\skal{[a^*c]}{[b]}_\Phi} 
  \le \seminorm[\big]{\Phi}{[a^*c]} \seminorm[\big]{\Phi}{[b]}$ holds for all $b\in\mathcal{A}_\CC$.
  So $\mathcal{A}_\CC \big/ \mathcal{G}_\Phi \ni [b] \mapsto \skal{[c]}{[ab]}_\Phi \in \CC$ is $\seminorm{\Phi}{\argument}$\=/continuous,
  which shows that $[c]$ is an element of the domain of the adjoint of $\mult{a}$.
  As $\mult{a}$ is essentially self-adjoint, $[c]$ is also an element of the domain
  of the closure of $\mult{a}$, so there exists a sequence $\big([b_n]\big)_{n\in \NN}$ in
  $\mathcal{A}_\CC \big/ \mathcal{G}_\Phi$ that is $\seminorm{\Phi}{\argument}$\=/convergent against $[c]$ and such that
  $\big([ab_n]\big)_{n\in \NN}$ is a $\seminorm{\Phi}{\argument}$\=/Cauchy sequence. Consequently, a corresponding
  sequence of representatives $(b_n)_{n\in \NN}$ in $\mathcal{A}_\CC$ is $\seminorm{\Phi}{\argument}$\=/convergent
  against $c$ and $(ab_n)_{n\in \NN}$ is a $\seminorm{\Phi}{\argument}$\=/Cauchy sequence, 
  hence $\seminorm{\Phi}{\argument}$\=/convergent against
  $ac$ by the previous Lemma~\ref{lemma:closable}. This shows $c\in \mathcal{A}^\cl_\CC$, 
  hence $\mathcal{A}^\cl_\CC = \mathcal{A}^\acl_\CC$.

  Finally, given $b\in \mathcal{A}^\cl_\CC$, then it only remains to show that 
  $(a+\I\lambda\Unit)^{-1}b \in \mathcal{A}^\cl_\CC$ for both $\lambda\in\{-1,+1\}$. 
  As $\mult{a}$ is essentially self-adjoint, the image of 
  $\mathcal{A}_\CC \big/ \mathcal{G}_\Phi$ under $\mult{a+\I\lambda\Unit}$ is dense in
  $\mathcal{A}_\CC \big/ \mathcal{G}_\Phi$, hence in $\Hilb_\Phi$, and so there exists
  a sequence $(c_n)_{n\in \NN}$ in $\mathcal{A}_\CC$ such that the sequence
  $\big(\mult{a+\I\lambda\Unit}[c_n]\big)_{n\in \NN}$ in $\mathcal{A}_\CC \big/ \mathcal{G}_\Phi$
  is $\seminorm{\Phi}{\argument}$\=/convergent against $[b] \in \Hilb_\Phi$. 
  It follows that the sequence $\big((a+\I\lambda\Unit)c_n\big)_{n\in \NN}$ in $\mathcal{A}_\CC$
  is $\seminorm{\Phi}{\argument}$\=/convergent against $b$ because $\seminorm{\Phi}{d} = \seminorm{\Phi}{[d]}$
  even for all $d\in \mathcal{A}^\cl_\CC$. The $\seminorm{\Phi}{\argument}$\=/continuity
  of the left multiplication with 
  $(a+\I\lambda\Unit)^{-1}$ now shows that the sequence $(c_n)_{n\in \NN}$ in $\mathcal{A}_\CC$
  is $\seminorm{\Phi}{\argument}$\=/convergent against $(a+\I\lambda\Unit)^{-1}b$, 
  so $(a+\I\lambda\Unit)^{-1}b\in\mathcal{A}^\cl_\CC$.
\end{proof}
\subsection{Application: Determinacy of Moment Problems}
For the rest of this section, let $\mathcal{I}$ be a Riesz ideal and unital subalgebra of $\Stetig(X)$
and $\Phi \colon \mathcal{I} \to \RR$ a positive linear map and strictly $\mathcal{I}$\=/continuous.
Moreover, let 
$\mathcal{I}_{\CC} \coloneqq \set{f_r + \I f_i \colon X \to \CC}{f_r,f_i \in \mathcal{I}}$ be 
the complexification of $\mathcal{I}$, which is a commutative $^*$\=/algebra of continuous 
complex-valued functions from $X$ to $\CC$ with pointwise complex conjugation as $^*$\=/involution.
Then $\Phi$ can be extended linearly to a positive Hermitian linear map $\Phi_\CC\colon \mathcal{I}_\CC\to \CC$
and the complex linear span $\mathcal{A}_\CC$ in $\mathcal{I}_\CC$ of every subalgebra 
$\mathcal{A}$ of $\mathcal{I}$ is a unital $^*$\=/subalgebra. This allows to apply
the notions and results for $^*$\=/algebras from the previous Section~\ref{sec:staralg}.
The main aim is to derive a sufficient condition for $\mathcal{A}_\CC$ to be dense in 
$\mathcal{I}_\CC$ with respect to the seminorm $\seminorm{\Phi}{\argument}$ constructed out 
of $\Phi_\CC$ like before.
\begin{proposition} \label{proposition:closures}
  Let $S_\CC$ be a $\seminorm{\Phi}{\argument}$-closed subset of $\mathcal{I}_\CC$, then 
  $S \coloneqq S_\CC \cap \mathcal{I}$ is $\mathcal{I}$\=/closed.
\end{proposition}
\begin{proof}
  Let $(s_n)_{n\in \NN}$ be a sequence in $S$ that is $\mathcal{I}$\=/convergent 
  against some $\hat{s} \in \mathcal{I}$, i.e. there exists a decreasing sequence
  $(f_k)_{k\in \NN}$ in $\mathcal{I}$ with pointwise infimum $0$ and such that for all $k\in \NN$
  there is an $N\in \NN$ with the property that $\abs{\hat{s}-s_n} \le f_k$ for all $n\in \NN$
  with $n\ge N$.
  Then the sequence $(f_k^2)_{k\in \NN}$ is also decreasing with pointwise infimum
  $0$, so $\NN \ni k \mapsto \Phi(f_k^2) \in \RR$ is a decreasing sequence in $\RR$ 
  with infimum $0$ because $\Phi$ is strictly $\mathcal{I}$\=/continuous by assumption.
  So given $\epsilon \in {]0,\infty[}$, then there exists a $k\in \NN$ such that 
  $\Phi(f_k^2) \le \epsilon^2$ and thus also an $N \in \NN$
  for which $
    \seminorm{\Phi}{\hat{s}-s_n}^2
    = 
    \Phi\big( (\hat{s}-s_n)^2 \big)
    \le 
    \Phi\big(f_k^2\big)
    \le
    \epsilon^2$
  for all $n\in \NN$ with $n\ge N$, i.e. $(s_n)_{n\in \NN}$ converges against $\hat{s}$
  with respect to $\seminorm{\Phi}{\argument}$ and thus $\hat{s} \in S$.
\end{proof}
In other words, the above Proposition~\ref{proposition:closures} shows that the
topology corresponding to the $\mathcal{I}$\=/closed subsets on $\mathcal{I}$ is stronger
than the (relative) $\seminorm{\Phi}{\argument}$\=/topology for all strictly $\mathcal{I}$\=/continuous
positive Hermitian linear maps $\Phi\colon \mathcal{I}\to \RR$.
This now allows to derive a variant of \cite[Thm.~14.2]{schmuedgen:TheMomentProblem}:
\begin{theorem} \label{theorem:determinacy}
  Let $\mathcal{A}$ be a unital subalgebra of $\mathcal{I}$ that contains a point-separating 
  subset $S$. Using the notatation introduced in Section~\ref{sec:staralg}, the following is equivalent:
  \begin{enumerate}
    \item $\mathcal{A}^\acl_\CC = \mathcal{I}_\CC$ for all $a\in S$.
    \label{enum:determinacy:1}
    \item For every $a\in S$ the symmetric operator $\mult{a}$ in the GNS-representation 
    of $\mathcal{A}_\CC$ with respect to $\Phi_\CC$ is essentially self-adjoint.
    \label{enum:determinacy:2}
  \end{enumerate}
  If one, hence both of these equivalent statements are true, then also $\mathcal{A}^\cl_\CC=\mathcal{I}_\CC$.
\end{theorem}
\begin{proof}
  Note that $a+\I\lambda \Unit$ is invertible in $\mathcal{I}_\CC$ for both $\lambda \in \{-1,+1\}$
  and all $a\in S$. So the first point implies the second due to the equivalence of \refitem{item:staralg:2}
  and \refitem{item:staralg:3} in Proposition~\ref{proposition:staralg}. 
  
  Conversely, assume that all $\mult{a}$ with $a\in S$ are essentially self-adjoint.
  Let $\mathcal{B}_{\CC}$ be the complex unital subalgebra of $\mathcal{I}_\CC$ that
  is generated by all $(a+ \I\lambda \Unit)^{-1}$ with $a\in S$ and both $\lambda\in \{-1,+1\}$,
  then $\mathcal{B}_{\CC}$ is a unital $^*$\=/subalgebra.
  From the equivalence of \refitem{item:staralg:3} and \refitem{item:staralg:1}
  in Proposition~\ref{proposition:staralg} it follows that 
  $\mathcal{A}^\acl_\CC = \mathcal{A}^\cl_\CC$ for all $a\in S$, so it only remains
  to show that $\mathcal{A}^\cl_\CC=\mathcal{I}_\CC$. Moreover, it also follows 
  $(a+\I\lambda)^{-1} b \in \mathcal{A}^\cl_\CC$ for all $a\in S$, both $\lambda \in \{-1,+1\}$
  and all $b\in \mathcal{A}^\cl_\CC$. This implies $\mathcal{B}_{\CC}\subseteq \mathcal{A}_\CC^\cl$
  because $\Unit\in\mathcal{A}_\CC^\cl$.  
  Now let $\mathcal{B} \coloneqq \mathcal{B}_{\CC}\cap\mathcal{I}$ 
  be the real unital subalgebra of Hermitian elements in $\mathcal{B}_{\CC}$.
  By construction, $\mathcal{B}$ consists of uniformly bounded functions 
  only. As $S$ is point-separating, $\mathcal{B}$ is also point-separating:
  Indeed, given $x,y\in X$ with $x\neq y$, then there exists $a\in S$ with
  $a(x)\neq a(y)$, thus $(a+\I\Unit)^{-1}(x) \neq (a+\I\Unit)^{-1}(y)$ and at 
  least one of the pointwise real or imaginary part of $(a+\I\Unit)^{-1}$,
  which are both elements of $\mathcal{B}$, separates $x$ and $y$.
  So $\mathcal{A}^\cl_\CC \cap \mathcal{I}$ contains the point-separating 
  unital subalgebra $\mathcal{B}$ of $\Stetig_b(X)$ and is $\mathcal{I}$\=/closed 
  by the previous Proposition~\ref{proposition:closures}. Thus 
  Theorem~\ref{theorem:convergeAlg} applies and shows that $\mathcal{A}^\cl_\CC \cap \mathcal{I}=\mathcal{I}$,
  therefore $\mathcal{A}_\CC^\cl = \mathcal{I}_\CC$.
\end{proof}
\begin{example}
  Let $\mathcal{I} \coloneqq \set[\big]{f\in \Stetig(\RR^n)}{\exists_{p\in \RR[x_1,\dots,x_n]} \forall_{x\in \RR^n}: \abs{f(x)} \le p(x)}$
  be the Riesz ideal and unital subalgebra of $\Stetig(\RR^n)$ of polynomially bounded functions. Let
  $S \coloneqq \{x_1,\dots,x_n\} \subseteq \mathcal{I}$ be the set of coordinate functions on $\RR^n$, 
  which separates the points of $\RR^n$ and generates the unital subalgebra $\mathcal{A}$ of polynomial functions.
  Given a positive Borel measure $\mu$ on $\RR^n$ for which all polynomial functions are integrable,
  then all functions in $\mathcal{I}$ are integrable with respect to $\mu$ and
  the positive linear map $\Phi \colon \mathcal{I} \to \RR$, $f \mapsto \Phi(f) \coloneqq \int_{\RR^n} f\,\D\mu$
  is strictly $\mathcal{I}$\=/continuous like in Example~\ref{example:integrals}.
  By the above Theorem~\ref{theorem:determinacy}, essential self-adjointness of all the multiplication operators 
  with $a \in \{x_1,\dots,x_n\}$ in the GNS-representation of $\mathcal{A}_\CC$
  with respect to $\Phi_\CC$ is equivalent to the ``strong determinacy'' of $\mu$.
  This reproduces \cite[Thm.~14.2]{schmuedgen:TheMomentProblem}. The 
  Stone-Weierstraß-type Theorem~\ref{theorem:convergeAlg} used here essentially replaces 
  the application of spectral theory there -- the operator theoretic considerations 
  are roughly the same.
\end{example}
Note that Theorem~\ref{theorem:determinacy} actually does give a sufficient condition
for two strictly $\mathcal{I}$\=/continuous positive Hermitian maps $\Psi,\Psi' \colon \mathcal{I} \to \RR$
to coincide: If $\Psi$ and $\Psi'$ coincide on a unital subalgebra
$\mathcal{A}$ of $\mathcal{I}$, then they also coincide with $\Phi \coloneqq (\Psi+\Psi')/2$
on $\mathcal{A}$.
Statement \refitem{enum:determinacy:2} in Theorem~\ref{theorem:determinacy}
only depends on the restriction of $\Phi$ to $\mathcal{A}$, hence can be checked
equivalently using one of $\Psi$ or $\Psi'$. If this is fulfilled, then
$\mathcal{A}_\CC^\cl = \mathcal{I}_\CC$, which implies that $\Psi$ and $\Psi'$
coincide on whole $\mathcal{I}$ because they are both $\seminorm{\Phi}{\argument}$\=/continuous.

\footnotesize
\renewcommand{\arraystretch}{0.5}

\end{onehalfspace}

\begin{thebibliography}{1}

\bibitem {jurzak:dominatedConvergenceAndStoneWeierstrass}
{Jurzak, J.-P.: }\newblock \emph{Dominated Convergence and
  Stone-Weierstrass Theorem}.
\newblock Journal of Applied Analysis  \textbf{11}.2 (2010), 207 -- 223.

\bibitem {kelley:GeneralTopology}
{Kelley, J.~L.: }\newblock \emph{General Topology}.
\newblock D. van Nostrand Company Inc, 1955.

\bibitem {luxemburg.zaanen:RieszSpacesI}
{Luxemburg, W. A.~J., Zaanen, A.~C.: }\newblock \emph{Riesz Spaces I}.
\newblock North-Holland publishing company, 1971.

\bibitem {ng.warner:ContinuityOfPositiveAndMultiplicativeFunctionals}
{Ng, S.-b., Warner, S.: }\newblock \emph{Continuity of positive and
  multiplicative functionals}.
\newblock Duke Math. J.  \textbf{39}.2 (1972), 281 -- 284.

\bibitem {schmuedgen:TheMomentProblem}
{Schm{\"u}dgen, K.: }\newblock \emph{The Moment Problem}.
\newblock Springer, 2017.


\end{thebibliography}
\end{document}